\documentclass[a4paper,reqno]{amsart}

\usepackage[a4paper,width={16cm},left=2.5cm,bottom=3cm, top=3cm]{geometry}
\usepackage{enumerate}
\usepackage{yfonts}

\usepackage[english]{babel}

\usepackage{a4wide}
\usepackage[dvipsnames]{xcolor}
\usepackage{amsmath}
\usepackage{amssymb}
\usepackage{amsfonts}
\usepackage{amsthm,graphicx}
\usepackage{comment}
\usepackage{csquotes}
\usepackage{esint}
\usepackage{hyperref}
\hypersetup{
	linktocpage=true,
	colorlinks=true,
	linkcolor=blue!70!black,
	citecolor=orange!40!red,
	urlcolor=magenta!80!black,
}

\usepackage{mathtools}
\mathtoolsset{showonlyrefs}

\numberwithin{equation}{section}
\newtheorem{theorem}{Theorem}[section]
\newtheorem{lemma}[theorem]{Lemma}
\newtheorem{corollary}[theorem]{Corollary}

\newtheorem{proposition}[theorem]{Proposition}
\theoremstyle{definition}  
\newtheorem{definition}[theorem]{Definition} 

\newtheorem{example}[theorem]{Example}
\newtheorem{remark}[theorem]{Remark}

\newcommand{\mc}{\mathcal}
\newcommand{\mb}{\mathbb}

\newcommand{\la}{\lambda}
\newcommand{\norm}[1]{\left\lVert#1\right\rVert}
\newcommand{\pd}[2]{\frac{\partial#1}{\partial#2}}

\newcommand{\R}{\mb{R}}

\newcommand{\e}{\varepsilon}

\newcommand{\Tr}{\mathop{\rm{Tr}}}
\newcommand{\dive}{\mathop{\rm{div}}}

\newcommand{\capa}[2]{\textnormal{\rm{Cap}}_{#1}(#2)}

\usepackage{subfig}
\usepackage{tikz}
\usepackage{xcolor}
\usepackage{pgfplots}
\pgfplotsset{compat=1.18}
\usepackage{mathrsfs}
\usetikzlibrary{arrows}
\usetikzlibrary{shadings}
\makeatletter

\begin{document}
\title[On  the blow-up of the vectorial Bernoulli free boundary problem]
{On  the blow-up of the vectorial Bernoulli free boundary problem}

\author[G. Siclari and B. Velichkov]{Giovanni Siclari and Bozhidar Velichkov}

\address{Bozhidar Velichkov
\newline \indent Dipartimento di Matematica
\newline \indent Universita di Pisa 
\newline\indent Largo Bruno Pontecorvo, 5, 56127 Pisa, Italy}
\email{bozhidar.velichkov@unipi.it}

\address{Giovanni Siclari 
\newline \indent Centro di Ricerca Matematica Ennio De Giorgi
\newline \indent Scuola Normale Superiore di Pisa
\newline\indent Piazza dei Cavalieri 3, 56126 Pisa, Italy}
\email{giovanni.siclari@sns.it}

\date{\today}

\begin{abstract}
In this paper, we complete the classification of  the blow-up limits of minimizers of the vectorial Bernoulli free boundary problem. Furthermore, we study the vectorial Bernoulli free boundary problem in a bounded box $D$, with a constraint $m$ on the measure of the positivity set, and the asymptotic of minimizers as the measure constraint $m$ tends to $|D|$. Such a study with a linear datum on the fixed boundary is the main ingredient  for the characterization of the singular homogeneous global solutions of the vectorial problem and, thus, for the classification of the blow-up limits.
\end{abstract}

\maketitle


\medskip 


\section{Introduction}\label{sec_intro}
Let $D$ be an open set in $\R^d$ and let $\Lambda>0$. For every $W\in H^1(D;\R^k)$ we define the vectorial functional 
\begin{equation}\label{e:vectorial-functional-in-D}
J_\Lambda(W,D):=\int_{D} |\nabla W|^2\, dx+\Lambda |\Omega_W\cap D|,
\end{equation}
where for any $W:D\to\R^k$ we set 
\begin{equation}\label{e:positivity-set-defintion}
\Omega_W:=\{W\neq 0\}.
\end{equation}
\begin{definition}[Vectorial minimizers in $D$]\label{def:minimizers-in-D}
Let $\Lambda\ge 0$. We say that $U\in H^1(D,\R^k)$ is a minimizer of the vectorial functional $J_\Lambda$ in the open set $D\subset\R^d$, if  
\begin{equation}\label{prob_min_vec}
J_\Lambda(U,D)\le J_\Lambda(W,D)\quad\text{for every}\quad W\in H^1_{loc}(D,\R^k)\quad\text{with}\quad W-U\in H^1_{0}(D,\R^k).
\end{equation}
If \eqref{prob_min_vec} holds, we will say that $U$ is a minimizer (or variational solution) of the vectorial Bernoulli problem.
\end{definition}
The properties of the variational solutions to the vectorial boundary problem have been subject of several studies in recent years (see the survey \cite{TV_survey} for a detailed overview).
Basic properties as Lipschitzianity and non-degeneracy of the minimizers $U$ of \eqref{prob_min_vec} and interior density estimates and local finiteness of the perimeter of $\Omega_U$ 
have been established in \cite{BMV_Lip,CSY_vectorial,KL_non_deg,KL_deg,MTV_epsilon,MTV_reg_spec, MTV_reg_vect,T_lip,T_Reg}.
Furthermore, letting  $U$ be a minimizers \eqref{prob_min_vec} and  $\Omega^{(\gamma)}_{U}$ be the set of all points having density $\gamma \in [0,1]$,  we can divide the free boundary into three parts:
\begin{align}\label{def_reg}
&\mathrm{Reg}(\partial \Omega_{U}) := \Omega^{(1/2)}_{U} \cap D,  \\  \label{def_sing2}
&\mathrm{Sing}_2(\partial \Omega_{U}) :=  \Omega^{(1)}_{U} \cap \partial  \Omega_{U} \cap D, \\  \label{def_sing1}
&\mathrm{Sing}_1(\partial \Omega_{U}) := (\partial  \Omega_{U} \cap D) \setminus 
\big( \mathrm{Sing}_2(\partial  \Omega_{U}) \cup \mathrm{Reg}(\partial  \Omega_{U}) \big).
\end{align}
\noindent{\it Regular part of the free boundary.} The regularity of the free boundary at points belonging to $\mathrm{Reg}(\partial \Omega_{U})$ has been established by reducing  the problem to the scalar case (as in \cite{CSY_vectorial,MTV_harn,MTV_reg_vect,MTV_reg_spec}), by an epiperimetric inequality (see \cite{SV_epi}), or via improvement of flatness methods applied directly in the vectorial setting (see \cite{DT_impr,KL_non_deg,KL_deg}). 

\noindent{\it One-phase singular set.} The analysis of the one-phase singular set $\mathrm{Sing}_1(\partial \Omega_{U})$ follows from the results about the singular set for the Alt-Caffarelli's problem (\cite{W_dim_red,CJK_cones_stab,JS_cones,EE_quant_strat}). Precisely, in \cite{CSY_vectorial,KL_deg,KL_non_deg,MTV_reg_spec} it was shown that the set $\mathrm{Sing}_1(\partial \Omega_{U})$ has Hausdorff dimension at most $d-5$ and consists of points at which $\partial\Omega_U$ has one-phase conical singularities, while the rectifiability of $\mathrm{Sing}_1(\partial \Omega_{U})$ follows from \cite{EE_quant_strat}.

\noindent{\it Two-phase singular set.} Concerning $\mathrm{Sing}_2(\partial \Omega_{U})$, we know that it might have dimension $d-1$; for instance, explicit examples of such solutions are available in the two-phase case $k=1$ (see \cite{DSV-quasi}). On the other hand, it is not known whether the set $\mathrm{Sing}_2(\partial \Omega_{U})$ is regular or even if it contained in a regular $d-1$ dimensional manifold. 
A key step in establishing such results is the classification of the blow-up limits at points of $\mathrm{Sing}_2(\partial \Omega_{U})$. Up to this point it was known that these blow-up limits are linear functions (see \cite{MTV_reg_vect}); on the other hand there was no general criterion establishing whether a given linear function can be a blow-up limit.\medskip

In the present paper (see Theorem \ref{theo_La_computed} below) we provide a full classification of the blow-up limits at points in $\mathrm{Sing}_2(\partial \Omega_{U})$. We start by recalling the following definition.
\begin{definition}[Blow-ups]\label{def:blow_up}
Letting $x_0 \in \partial \Omega_{U}$, we say that $U_0 \in H^1_{loc}(\R^d,\R^k)$ is a blow-up for $U$ at $x_0$ is there exists a sequence $r_n \to 0^+$ such that  
$\frac{1}{r_n}U(x_0+r_nx) \to U_0(x)$ uniformly on compact sets in $\R^d$.
We indicate with $\mc{BU}_U(x_0)$  the sets of all the blow-ups of $U$ at $x_0$.
\end{definition}
It was proved in  \cite[Section 2D]{MTV_reg_vect} that, if $x_0\in \mathrm{Sing}_2(\partial \Omega_{U})$, then any  $U_0 \in \mc{BU}_U(x_0)$  is a linear map, that is, 
\begin{equation}
U_0(x)= A x \quad \text{ where } A=(a_{i,j}) \in \R^{k,d},
\end{equation}
$\R^{d,k}$ being the space of $k\times d$ real matrices, and a global minimizer  of the vectorial Bernoulli free boundary problem. We recall the following definition.
\begin{definition}[Global vectorial minimizers]\label{def:global-solution}
Let $\Lambda\ge 0$. We say that $U\in H^1_{loc}(\R^d,\R^k)$ is a global minimizer of the vectorial functional $J_\Lambda$, if for any $R>0$, $U$ is a minimizer of the vectorial Bernoulli problem in $B_R$.  
\end{definition}
\noindent{\it Stratification of $\mathrm{Sing}_2(\partial \Omega_{U})$ and the rank of the blow-up matrix.} In \cite{MTV_reg_vect} it was proved that the rank of $A \in \mc{BU}_U(x_0)$ depends only on $x_0\in \mathrm{Sing}_2(\partial \Omega_{U})$ thus showing that the sets 
\begin{equation}
S_j:=\{x_0 \in\mathrm{Sing}_2(\partial \Omega_{U}): \mathop{\rm{rk}}(A)=j, A\in \mc{BU}_U(x_0)\}.
\end{equation}
are well-defined and disjoint thus leading to the stratification
\begin{equation}
\mathrm{Sing}_2(\partial \Omega_{U})=\bigcup_{j=1}^d  S_j.
\end{equation}
Furthermore, combining the results obtained in \cite{MTV_reg_vect} with \cite{PESV_rec}, we have that the stratum $S^{j}$ is $d-j$ rectifiable and has local finite $(d-j)$-Hausdorff  measure. \medskip

\noindent{\it About the norm of the blow-up matrix.} Again in \cite{MTV_reg_vect} it was shown that, letting 
\begin{equation}
\|A\|^2:=\sum_{i,j}a_{i,j}^2,
\end{equation}
it holds true that if 
\begin{equation}
\|A\|^2\ge \Lambda,
\end{equation}
then $Ax$ is a global minimizer of $J_\Lambda$. Furthermore,  if $\mathop{\rm{rk}}(A)=1$, then also the converse holds true, that is, 
if $Ax$ is a global minimizer of $J_\Lambda$ then 
$$\|A\|^2\ge \Lambda.$$

It has been an open problem since \cite{MTV_reg_vect} whether this second inequality holds in general; this question, in particular, is fundamental for the implementation of viscosity improvement-of-flatness techniques in the spirit of \cite{D_impr,DT_impr,MTV_epsilon} (see for instance the survey \cite{TV_survey} and, more precisely, \cite[Open problem 5.5]{TV_survey}). 
In order to give an answer to this question, we 
compute for any linear map $Ax$ the optimal, that is, the  biggest,  constant $\Lambda>0$ such that $Ax$ is a global minimizer of $J_\Lambda$ in the sense of Definition \ref{def:global-solution}.  Precisely, for any matrix $A\in\R^{k,d}$, we define the quantity 
\begin{equation}\label{def_La}
\Lambda^*(A):=\sup\{\Lambda>0:Ax \text{ is a global minimizer of } J_\Lambda\}.
\end{equation}

Surprisingly, from our analysis (see the general Theorem \ref{theo_La_computed} and the example in Theorem \ref{theo_example}) it follows that $\Lambda^*(A)$ raises above the threshold $\|A\|^2$ as soon as ${\rm{rk}}(A)>1$, which means that, for any $\Lambda>0$ and in any dimension $d\ge 2$, there are global minimizers $Ax$ of $J_\Lambda$ with $\|A\|^2<\Lambda$.
 A key observation is that we can characterize  $\Lambda^*(A)$ in terms of a variational free boundary capacitary problem in $\R^d$.
In this direction, letting
\begin{equation}
D^{1,2}(\R^d, \R^k):=\{W \in H^1_{loc}(\R^d,\R^k): \nabla W \in L^2(\R^d,\R^{d,k})\},
\end{equation}
our   main result is  the following. 
\begin{theorem}\label{theo_La_computed}
Let $A$ be a $k \times d$ matrix with ${\rm{rk}}(A)=n$ and  $1\le n\le d$. Let $Q \in \R^{d,d}$ be an orthogonal matrix such that 
$A=\begin{bmatrix} A_1, 0\end{bmatrix}Q$, for some matrix $A_1 \in \R^{k,n}$ of rank $n$.
Then, the following holds: 
\begin{itemize}
\item if ${\rm{rk}}(A)>1$, then $\Lambda^*(A)$ can be characterized as
\begin{equation}\label{eq_theo_La_computed}
\Lambda^*(A)= \inf\left\{\int_{\R^n} |\nabla W|^2\, dy:W\in D^{1,2}(\R^n, \R^k),  |\{W(y)= A_1y\}|=1\right\},
\end{equation}
and, in particular, 
\begin{equation}\label{eq_La_rk1}
\Lambda^*(A)>\norm{A}^2;
\end{equation}
\item if  ${\rm{rk}}(A)=1$, then
\begin{equation}\label{eq_La_rk_bigger_1}
\Lambda^*(A)=\norm{A}^2.
\end{equation}
\end{itemize}
\end{theorem}
In order to prove the theorem above, we first consider the case $\mathrm{rk}(A)=d$ (see Section \ref{sub_minezers_properties}), for which we prove the existence of a function $W$ that minimizes \eqref{eq_theo_La_computed} and has support localized around the origin (see Theorem \ref{theo_V}). When $\mathrm{rk}(A)<d$, the variational problem in $\R^d$ might not admit minimizers (see Remark \ref{reamrk_existence_non_existence_min}). 
We deal with this degenerate case via a dimension reduction argument that allows to bring the problem back to the case of a matrix of maximal rank.\medskip

The functional $\Lambda^\ast$ allows to define a new type of energy density for the vectorial problem. Precisely, if  $x_0 \in \mathrm{Sing}_2(\partial \Omega_{U})$  we show that $\Lambda^*(A)$ depends only on $x_0 \in \partial \Omega_U$ and not on the blow-up $A \in \mc{BU}_U(x_0)$, which a priori depends on the blow-up sequence of rescaled functions. It also turns out that this quantity is upper semicontinuous with respect to the variable $x_0$. 
\begin{theorem}\label{theo_dependence_LA}
Let $U$ be a minimizers of \eqref{prob_min_vec} and  $x_0 \in \mathrm{Sing}_2(\partial \Omega_{U})$. The constant $\Lambda^*(A)$  depends only on $x_0$ and not on $A \in \mc{BU}_U(x_0)$. Furthermore, the function $x_0 \mapsto \Lambda^*$  is upper semicontinuous. 
\end{theorem}
Given the relevance of the constant $\Lambda^*(A)$, in the next subsection we discuss more in detail the minimization problem in the right hand side of \eqref{eq_theo_La_computed}.

\subsection{Existence and properties of minimizers}\label{sub_minezers_properties}
In order to prove Theorem \ref{theo_La_computed} we first answer some natural questions about the existence of minimizers for the problem  \eqref{eq_theo_La_computed} and their properties. For any $R>0$, we consider  the following minimization problem with measure constraint 
\begin{equation}\label{prob_min_measure_A}
\inf\left\{\int_{B_R} |\nabla W|^2\, dx:W\in  H^1(B_R, \R^k), W-g \in  H_0^1(B_R, \R^k), |\Omega_W|=1\right\},
\end{equation}
for some  boundary datum $g \in H^1(D,\R^k)$ (the regularity of minimizers of \eqref{prob_min_measure_A} is dealt with  in Theorem \ref{theo_reg_m_fixed} below).
\begin{theorem}\label{theo_V}
Let $A$ be a $k \times d$ matrix with ${\rm{rk}}(A)=d$. 
Then 
\begin{equation}\label{prob_min_boundedness}
\inf\left\{\int_{\R^d} |\nabla W|^2\, dx:W\in D^{1,2}(\R^d, \R^k),  |\{W(x)= Ax\}|=1\right\}
\end{equation}
admits a minimizer $V \in D^{1,2}(\R^d,\R^k)$.
Furthermore, the contact set $\{V(x)=Ax\}$ is bounded. In particular, $V(x)-Ax$ is a minimizer of \eqref{prob_min_measure_A}  with boundary datum $V(x)=Ax$ in $B_R$, for $R$ large enough. 
Finally,  $V$ is globally Lipschitz, bounded  and, if  $\{V(x)=Ax\}\subset B_R$,
\begin{equation}\label{limit_V_infty}
|V|\le C|x|^{2-d} \quad \text{ in } \R^d \setminus B_R
\end{equation}
for some positive constant $C>0$ depending only on $d$, $R$ and $\norm{V}_{L^\infty(\partial B_ R)}$.
\end{theorem}

\begin{remark}\label{reamrk_existence_non_existence_min}
Without the assumption ${\rm{rk}}(A)=d$, we cannot in general expect minimizers of \eqref{prob_min_boundedness} to exist. In fact, the contact set, that is $\{W(x)=Ax\}$, of a competitor that nearly attains the infimum in \eqref{prob_min_boundedness}, should lie close to the null space of $A$. Because this set is a linear subspace, the contact sets of a minimizing sequence should become increasingly \textit{squeezed} along $\ker A$. In this regime, one should naturally anticipate pointwise convergence to $0$ which prevents the existence of minimizers. This phenomenon is ruled out when we  impose ${\rm{rk}}(A)=d$. In that case, the same heuristic suggests that the contact set is instead concentrated near $\{0\}$, and in fact in Theorem \ref{theo_V} we prove that it is bounded.
\end{remark}

\begin{remark}
It would be interesting to investigate further the shape of the contact set, e.g. $\{V(x)=Ax\}$,  in \eqref{prob_min_boundedness} at least in a  simple case as $A=\rm{Id}_d$. However, this seems to be a non-trivial open problem, since the techniques used in related problems, like the obstacle problem see for example \cite{EFW,ESW,S_null_quad_doma,SW}, do not seem to be applicable in this vectorial setting.  Other more classical approaches to obtain symmetry, like the moving plane method, also do not appear to yield any simple result.
\end{remark}

\subsection{An explicit bound for $\Lambda^*({\rm{Id}_d})$}\label{sub_id}

In the next theorem, we provide a quantitative bound from below on the quantity $\Lambda^*(A)-\|A\|^2$ in the case $A=\rm{Id}_d$ and $d\ge 3$.

\begin{theorem}\label{theo_example}
If $d\ge 3$ then
\begin{equation}
\Lambda^*({\rm{Id}_d}) -\|{\rm{Id}_d}\|^2\ge (d-1)^2-d.
\end{equation}
\end{theorem}

The main lemma used to prove Theorem \ref{theo_example} is the following radiality result, which we prove by a reflection argument in any direction and which can be used in general problems involving Sobolev capacities.
\begin{lemma}[Radiality reduction lemma]\label{lemma_rad}
Suppose that $f:B_1 \to \R$ is radial, $f \in H^1(B_1)$, and that $g:[0,|B_1|] \to [0,+\infty)$,   is  convex with $g(0)=0$ and $g>0$ in $(0,|B_1|]$. Then
\begin{multline}\label{eq_rad}
\inf\left\{\frac{\int_{B_1} |\nabla w|^2 \, dx}{g(|\{w=f\}|)}:  w\in H^1_0(B_1),|\{w=f\}|\neq 0\right\}\\
=\inf\left\{\frac{\int_{B_1} |\nabla w|^2 \, dx}{g(|\{w=f\}|)}: w\in H^1_0(B_1), |\{w=f\}|\neq 0, w \text{ radial}\right\}.
\end{multline}
\end{lemma}

\subsection{Strategy of the proof and plan of the paper}
Our strategy to prove  the results stated up to this point is based on the fact that the optimal constant $\Lambda^*(A)$ can be characterized as 
\begin{multline}\label{eq_La}
\Lambda^*(A)=\inf\left\{\frac{\int_{B_1} |\nabla W|^2 \, dx}{|\{W(x)=Ax\}|}: W \in H^1_0(B_1,\R^k),|\{W(x)=Ax\}|\neq 0\right\}\\
=\inf_{\e\in(0, |B_1|)}\inf\left\{\frac{1}{\e}\int_{B_1} |\nabla W|^2 \, dx: W \in H^1_0(B_1,\R^k),|\{W(x)=Ax\}|=\e\right\},
\end{multline}
taking a comparators of the form $W+Ax$ with $W \in H_0^1(B_1,\R^k)$ in 
\begin{equation}\label{prob_min_B1}
\inf\left\{\int_{B_1} |\nabla W|^2\, dx+\Lambda^*(A) |\Omega_W|:W\in  H^1(B_1, \R^k), W-Ax \in  H_0^1(B_1, \R^k)\right\},
\end{equation}
see the proof of Proposition \ref{prop_lim_inf_We} for details. In the same proposition we are going to show that the infimum over $\e\in (0, |B_1|)$ in \eqref{eq_La} actually coincides with the limit as $\e\to 0^+$.
Then, going back for any fixed $\e\in (0, |B_1|)$ to the problem
\begin{equation}
\inf\left\{\int_{B_1} |\nabla W|^2 \, dx:W\in  H^1(B_1, \R^k),  W-Ax \in H^1(B_1,\R^k), |\Omega_W|=|B_1|-\e\right\},
\end{equation}
we can see that  computing $\Lambda^*(A)$ is linked to study minimizers  of  a measure constraint free boundary problem  as  $m:=|B_1|-\e \to |B_1|$, this approach  is inspired by \cite{NSV}, where such an analysis was carried out for a Dirichlet eigenvalues problem.

In this direction, also motivated by Theorem \ref{theo_V}, with greater generality on the domain and the boundary datum, we prove several results. 
More precisely, let $D \subset \R^d$ be a bounded, connected  open set, $g \in H^1(D, \R^k)$ and 
let us consider the vectorial  measure constrained  minimization problem
\begin{equation}\label{prob_min_measure}
\inf\left\{\int_{D} |\nabla W|^2\, dx:V\in  H^1(D, \R^k), W-g \in  H_0^1(D, \R^k), |\Omega_W|=m\right\},
\end{equation}
where $m \in (0,|D|)$. We are interested in the regularity of solutions $U_m$ of \eqref{prob_min_measure}, of the correspondent free boundaries and the uniformity of the regularity with respect to $m$ as $m$ approaches $|D|$.
Our first main result is interior regularity for fixed $m$.

Letting $\mathrm{Reg}(\partial \Omega_{U_m}), \mathrm{Sing}_2(\partial \Omega_{U_m})$ and $\mathrm{Sing}_1(\partial \Omega_{U_m})$ as in \eqref{def_reg}, \eqref{def_sing2} and \eqref{def_sing1} respectively, 
we have the following theorem.
\begin{theorem}\label{theo_reg_m_fixed}
There exists a solution to problem \eqref{prob_min_measure}. Any solution 
$U_m \in H^1(D, \mathbb{R}^k)$ is a locally  Lipschitz continuous function in $D \subset \mathbb{R}^d$. The set $\Omega_{U_m}$ has  locally finite 
perimeter in $D$ while the free boundary $\partial \Omega_{U_m} \cap D$ is the union of three disjoint sets: a regular part 
$\mathrm{Reg}(\partial \Omega_{U_m})$, a (one-phase) singular set 
$\mathrm{Sing}_1(\partial \Omega_{U_m})$ and a set of two-phase singular set 
$\mathrm{Sing}_2(\partial \Omega_{U_m})$. Furthermore:
\begin{enumerate}
\item the regular part $\mathrm{Reg}(\partial \Omega_{U_m})$ is an open subset of $\partial \Omega_{U_m}$ and is locally the graph of a $C^\infty$ function;
\item the one-phase singular set $\mathrm{Sing}_1(\partial \Omega_{U_m})$ consists only of points in which the Lebesgue density of $U$ is strictly between 
    $\frac12$ and $1$. Moreover, there is a $d^* \in \{5,6,7\}$ such that:
    \begin{itemize}
        \item if $d < d^*$, then $\mathrm{Sing}_1(\partial \Omega_{U_m})$ is empty;
        \item if $d = d^*$, then the singular set 
        $\mathrm{Sing}_1(\partial \Omega_{U_m})$ contains at most a finite number of 
        isolated points;
        \item if $d > d^*$, then the $(d-d^*)$-dimensional Hausdorff measure 
        of $\mathrm{Sing}_1(\partial \Omega_{U_m})$ is locally finite in $D$;
    \end{itemize}
    \item the set of two-phase singular set $\mathrm{Sing}_2(\partial \Omega_{U_m})$ is a 
    closed set of locally finite $(d-1)$-Hausdorff measure in $D$ and 
    consists only of points in which the Lebesgue density of $U$ is $1$ 
    while the blow-up limits are linear functions.
\end{enumerate}
\end{theorem}
To prove the theorem above,  we  study  minimizers of a related one parameter family of  penalized functionals  and then recover  regularity for solutions to \eqref{prob_min_measure} taking the parameter  small enough. This is a well-established strategy to deal with measure constrained free boundary problems, see for example  \cite{AAC_classical,MST_spectral_partition}.

Next we study the asymptotics  and the uniformity of the regularity of minimizers with respect to $m$ as $m$ approaches $|D|$. For the sake of clarity we state the result with respect to the parameter $\e:=|D|-m$.

\begin{theorem}\label{theo_reg_m_var}
Let $\{U_\e\}$ be a family of minimizers of \eqref{prob_min_measure} with $m:=|D|-\e$ and let $h:D\to \R^k$ be the harmonic extension of the boundary datum $g$ to $D$. Assume that for some $x_0 \in D$
\begin{equation}\label{hp_h}
h(x_0)=0, \quad h(x)\neq 0 \text{ if } x\neq x_0\quad \text{ and } \quad   \mathop{\rm{Ker}}(\nabla h(x_0))=\{ 0\}.
\end{equation}
Then  there exists $\e_0>0$  such that $\{U_\e\}_{\e\in (0,\e_0]}$ are  locally equi-Lipschitz in $D$ and 
\begin{equation}\label{limit_Ue_h}
U_\e \to h \quad \text{ strongly in  }H^1(D,\R^k)\cap C^{0,\alpha}_{loc}(D,\R^k) \text{ as } \e \to 0^+, 
\end{equation}
for any $\alpha \in (0,1)$. Furthermore
\begin{multline}\label{limit_Ue_h_precise}
\lim_{\e \to 0^+} \frac{\int_{D} |\nabla U_\e|^2 -|\nabla h|^2 \, dx}{\e}\\
=\min\left\{\int_{\R^d} |\nabla W|^2\, dx:W\in D^{1,2}(\R^d, \R^k),  |\{W(x)= \nabla h(x_0)x\}|=1\right\}
\end{multline}
and  the limit in local Hausdorff sense in $D$ of  $D \setminus \Omega_{U_\e}$ is $\{x_0\}$.
\end{theorem}

\begin{remark}
If $h$ does not vanish in $B_1$, it is not true in general that $\{U_\e\}$ are locally equi-Lipschitz in $D$ as Example \ref{exam_not_lip} shows.
On the other hand, if $h$ has more than one zero the situation is more complicated and need further investigation, in particular if the null set of $h$ is large.
It is reasonable to expect a result similar to Theorem \ref{theo_reg_m_var} if the null set of $h$ is discrete while the situation may be quite different if  the null set of $h$ has a bigger Hausdorff dimension. 
\end{remark}

The paper is organized as follow.  In Section \ref{sec_Basic} we study the regularity of \eqref{prob_min_measure}  for any fixed $m$ collecting some basic properties and introducing a one-parameter family of penalized functionals. In Section \ref{sec_shape_var} we prove regularity  for minimizers of the penalized functional and  then recover  regularity for solutions to \eqref{prob_min_measure} taking the parameter  small enough thus proving Theorem \ref{theo_reg_m_fixed}. In Section \ref{sec_uni}, by the means of a suitable rescaling, we prove the existence of minimizers of \eqref{prob_min_measure_A} and show how we can obtain Theorem \ref{theo_reg_m_var} as a consequence. In Section \ref{sec_linear} we focus on the linear case completing the proof of Theorem \ref{theo_La_computed} and Theorem \ref{theo_V} and finally we prove Theorem \ref{theo_dependence_LA}. Finally, in  Section \ref{sec_example} we prove Lemma \ref{lemma_rad} from which we obtain Theorem \ref{theo_example} as an easy corollary.

\section{Existence and penalization for a measure constrained vectorial free boundary problem}\label{sec_Basic}
In this section we prove existence of minimizers of \eqref{prob_min_measure} and  some basic properties. Then we  study the regularity of  minimizers of a related one parameter family of  penalized functionals.

\subsection{Existence} \label{subsec_existence}
Let us consider the auxiliary problem
\begin{equation}\label{prob_min_measure_relaxed}
\inf\left\{\int_{D} |\nabla W|^2\, dx:W\in  H^1(D, \R^k), W-g \in  H_0^1(D, \R^k), |\Omega_W|\le m\right\}.
\end{equation}
Then we have the following result.
\begin{proposition}\label{prop_minim}
For any $m \in (0, |D|)$,
\begin{enumerate}[(i)]
\item there exists a solution to \eqref{prob_min_measure_relaxed};
\item $U$ is a solution of \eqref{prob_min_measure} if and only if $U$ is a solution of \eqref{prob_min_measure_relaxed};
\item $u_i$ is harmonic on $\Omega_{U}$ for any $i=1,\dots,k$;
\item $|U|$ is subharmonic on $D$ in a distributional sense.
\end{enumerate}
\end{proposition}
\begin{proof}
If $\{U_n\}_{n \in \mathbb{N}}$ is a minimizing sequence, that clearly $\{U_n\}_{n \in \mathbb{N}}$ is bounded 
in $H^1(D, \R^k)$. In particular, up to subsequences, $U_n \rightharpoonup V_\infty$ weakly in $H^1(D, \R^k)$ as $n \to \infty$ for some $U_\infty \in H^1(D)$. Furthermore,
\begin{equation}
\chi_{\Omega_{U_{\infty}}} \le \liminf_{n\to \infty}\chi_{\Omega_{U_n}}
\end{equation}
and so $|\Omega_{U_{\infty}}| \le m$. Hence, $U_\infty$ is a minimum of  \eqref{prob_min_measure_relaxed}.

In order to prove (ii) it is enough to show that any minimizer of   \eqref{prob_min_measure_relaxed} satisfies $\Omega_U=m$. We argue by contradiction,
supposing that $|\Omega_U|<m$. Then there exists a radius $r_0>0$ such that $|B_{r_0}|<m-|\Omega_U|$. For any $x \in D$, 
let $h_i$ be the harmonic extension of $u_i$ on $\partial B_r(x)$ where $r \in (0,\min{r_0, d(x,\partial D)}$.
We may consider the competitor $\widetilde{h}=(\widetilde{h}_1,\dots, \widetilde{h}_k)$ defined as 
\begin{equation}
\widetilde{h}_i=
\begin{cases}
u_i \text{ in } D\setminus B_r(x), \\
h_i  \text{ in }  B_r(x).
\end{cases}
\end{equation}
By harmonicity of $h$ and minimality of $U$, we conclude that for any 
\begin{equation}
\int_{B_r(x)} |\nabla(U-h)|^2=0. 
\end{equation}
Hence, $U=h$ on $B_r(x)$. Since $x$ was an arbitrary point of  $D$, it follows that $U$ is harmonic on $D$ which contradict the fact that $|\{U \neq 0\}|\le m$.

Let $w \in H^1(D)$ with $w-u_i \in H_0^1(\Omega_U)$ for some $i=1,\dots, k$. Then $(u_1,\dots,w,\dots,u_k)$ is a comparator for $U$ since $\Omega_V \subset \Omega_U$. Hence, $U$ is harmonic on $\Omega_U$ by minimality.

Furthermore,  on $\Omega_U$,
\begin{multline}
\Delta |U|= \dive(\nabla |U|)=\sum_{i=1}^k \Delta u_i \frac{u_i}{|U|}  +\sum_{i=1}^k\nabla \left(\frac{u_i}{|U|}\right) \cdot \nabla u_i\\
= \frac{1}{|U|^3} \left[|U|^2\sum_{i=1}^k|\nabla u_i|^2 -\sum_{i,j=1}^ku_iu_j \nabla u_i \nabla u_j \right]=\frac{1}{|U|^3} 
\left[\sum_{i=1}^k\sum_{j=1,j\neq i}^k(-1)^j u_j \nabla u_i \right]^2 \ge 0,
\end{multline}
that is, $|U|$ is subharmonic  on $\Omega_U$.
Let $\varphi \in C^\infty_c(D)$, $\varphi \ge0$. Let us define
\begin{equation}
p_\e(x):=
\begin{cases}
0,  &\text{ if } x \in [0,\e/2],\\
\frac{1}{\e} (2x-\e),  &\text{ if } x \in [\e/2, \e],\\
1,  &\text{ if } x \in [\e, +\infty),\\
\end{cases}
\end{equation}
and $u_{\e,t}:= |U|-t p_\e(|U|)\varphi$  for any  $\e, t >0$. Then $\Omega_{u_{\e,t}} \subset \Omega_U $ and $U_{\e,t} \le |U|$. By subharmonicity of $|U|$ on $\Omega_U$,
\begin{equation}
\int_{D} |\nabla |U||^2 dx \le \int_{D} |\nabla u_{\e,t}|^2 dx.
\end{equation}
Furthermore,
\begin{equation}
|\nabla u_{\e,t}|^2= |\nabla |U||^2 -2t( |\nabla |U||^2 p_\e'(|U|) \varphi +p_\e(|U|)\nabla |U| \cdot \nabla \varphi) +o(t), \text{ as } t \to 0^+. 
\end{equation}
Hence,
\begin{equation}
\int_{D}p_\e(|U|)\nabla |U| \cdot \nabla \varphi \, dx \le -\int_{D} |\nabla |U||^2 p_\e'(|U|) \varphi dx \le 0,
\end{equation}
since $p_\e$ is increasing.
Passing to the limit as $\e  \to 0^+$, we conclude that $|U|$ is subharmonic in a distributional sense on $D$.
\end{proof}

\begin{remark}\label{remark_U_pointwise_bounded}
Since $|U|$ is subharmonic, for any $x \in D$, the maps
\begin{equation}
r \mapsto \fint_{B_r(x)} |U| \, dy \text{ and } r \mapsto \fint_{\partial B_r(x)} |U| \, d\mc{H}^{d-1} 
\end{equation}
are decreasing and so we may  define the pointwise value of $|U|$ as 
\begin{equation}
|U|(x):=\lim_{r \to 0^+}\fint_{B_r(x)} |U| \, dy=\lim_{r \to 0^+}\fint_{\partial B_r(x)} |U| \, d\mc{H}^{d-1}.
\end{equation}
In particular,  $U \in L^\infty_{loc}(D,\R^k)$ since if $d(x,\partial D)>\delta$ for some $\delta>0$
\begin{equation}
|U(x)| \le |B_\delta(x)|^{-1}\int_{B_\delta(x)} |U| \, dy \le |B_\delta|^{-1}\int_{D} |U| \, dy.
\end{equation}
\end{remark}

\subsection{A penalized functional}

In this subsection we establish some basic properties, as existence, Lipschitzianity, non-degeneracy and finiteness of the perimeter for minimizers of a penalised functional. Since the proof of this properties does not differ in any way from some already established results in the literature, we will simply recall the precise references and keep  track of all the constants involved.

Let us define for any $\eta\in (0,1]$ the Lipschitz function $f_{m,\eta}:\R \to \R,$
\begin{equation}\label{def_f_meta}
f_{m,\eta}(t):=
 \begin{cases}
\frac{1}{\eta}(t-m), &\text{ if } t >m,\\
\eta(t-m), &\text{ if } t \le m.
\end{cases}
\end{equation}

The following lemma is an easy consequence of the definition of $f_{m,\eta}$.
\begin{lemma}\label{lemma_f_eta}
We have that:
\begin{enumerate}[(i)]
\item $f_{m,\eta}(t) \ge -\eta m$  for any $t\ge0$,
\item  $\eta(t_2-t_1) \le f_{m,\eta}(t_2)-f_{m,\eta}(t_1) \le \frac{1}{\eta}(t_2-t_1) $  for any $0 \le t_1\le t_2$.
\end{enumerate}
\end{lemma}

Let us define the functional
\begin{equation}\label{def_J_eta}
J_{m,\eta}:H^1(D) \to [0,+\infty), \quad J_{m,\eta}(W):=\int_{D} |\nabla W|^2\, dx + f_{m,\eta} (|\Omega_W|).
\end{equation}
and  the energy level 
\begin{equation}\label{prob_min_J_meta}
c_{m,\eta}:=\inf\{J_{m,\eta}(W):  W\in  H^1(D, \R^k), W-g \in  H_0^1(D, \R^k)\}.
\end{equation} 
Let $h$ be the harmonic extension of $g$ in $D$.
\begin{proposition}\label{lemma_estimtes_c_eta}
For any $\delta>0$ there exists a constant $C>0$ depending only on $g, D, \delta$ such that 
\begin{equation}\label{inq_c_meta}
\int_D |\nabla h|^2 \, dx -m\le c_{m,\eta}  \le C \quad \text{ for any } m \in [\delta, |D|).
\end{equation}
\end{proposition}
\begin{proof}
Let $\delta>0$ and  $\varphi\in C^\infty(\overline D)$ be a cut off function such that $\varphi \equiv 1$ in a neighbourhood of $\partial D$  and $|\Omega_{\varphi}|\le\delta$. Then, by \eqref{prob_min_J_meta},
\begin{equation}
c_{m,\eta} \le \int_D |\nabla( \varphi h)|^2 \, dx +\eta (|\Omega_{\varphi}|-\delta).
\end{equation}
Hence, we have proved the upper estimate in \eqref{inq_c_meta}.
The lower estimate is a simple consequence  Lemma \ref{lemma_f_eta} and the harmonicity of $h$.
\end{proof}

With the same arguments exposed in Proposition \ref{prop_minim}, since $f_{m,\eta}$ is monotone, for any  $\eta\in (0,1]$ and $m \in (0,|D|)$, we can show the existence of a minimizer $U_{m,\eta} \in H^1(D)$ of  problem \eqref{prob_min_J_meta}.
Furthermore, by Lemma \ref{lemma_estimtes_c_eta},
\begin{equation}\label{ineq_U_meta_H1_bounded}
\int_{D} |\nabla U_{m,\eta}|^2 \, dx \le c_{m,\eta} +|D| \le C+|D|.
\end{equation}

\begin{remark}\label{remark_ineq_Omega_V}
We notice that \eqref{prob_min_J_meta} implies that for any $W\in  H^1(D, \R^k)$ with $W-g \in  H_0^1(D, \R^k)$ and such that $ |\Omega_W|\le |\Omega_{U_{m,\eta}}|$,
\begin{equation}
\int_{D} |\nabla U_{m,\eta}|^2 \, dx +\eta |\Omega_{U_{m,\eta}}|\le \int_{D} |\nabla W|^2 \, dx +\eta |\Omega_W|
\end{equation}
if $|\Omega_{U_{m,\eta}}| \le m$ or, if instead $|\Omega_{U_{m,\eta}}| > m$,
\begin{equation}
\int_{D} |\nabla U_{m,\eta}|^2 \, dx +\frac{1}{\eta} |\Omega_{U_{m,\eta}}|\le \int_{D} |\nabla W|^2 \, dx +\frac{1}{\eta} |\Omega_W|.
\end{equation}
Furthermore,  $U_{m,\eta}$ is actually a solution of \eqref{prob_min_measure} with the measure constraint $|\Omega_W|=|\Omega_{U_{m, \eta}}|$ instead of $|\Omega_W|=m$. Hence,  Proposition \ref{prop_minim} holds for $U_{m,\eta}$ 
and so $U_{m,\eta}$ is also a solution of \eqref{prob_min_measure_relaxed}.
\end{remark}

Let for any $\delta>0$
\begin{equation}\label{def_D_delta}
D_\delta:=\{x \in D: d(x, \partial D) >\delta\}.
\end{equation}

\begin{proposition}(Lipschitzianity) \label{prop_Lip}
For any $i=1, \dots, k$, the components $u_{i, m, \eta}$ of any minimizer $U_{m, \eta}$ are locally Lipschitz continuous in $D$. More precisely
\begin{equation}\label{ineq_estimate_lip_u_imeta}
\norm{\nabla U_{m,\eta}}_{L^\infty(D_\delta,\R^k)} \le \frac{C_d}{\eta}\left(1+\frac{\norm{U_{m,\eta}}_{H^1(D,\R^k)}+\norm{h}_{H^1(D,\R^k)}}{ \delta^{d+1}}\right),
\end{equation}
for some dimensional constant $C_d>0$.
\end{proposition}
\begin{proof}
Let  any $i\in \{1, \dots, k\}$ and $w\in  H^1(D)$, with $v-u_{i, m, \eta} \in  H_0^1(D)$. By Lemma \ref{lemma_f_eta}, taking as a comparator 
$W:=(u_{1, m, \eta},\dots, w, \dots,u_{k, m, \eta})$, we can see that for any $x \in D$ and any $r\in (0, d(x,\partial \Omega))$ the minimizer $U_{m,\eta}$ satisfies
\begin{equation}
\int_{B_r(x)} |\nabla u_{i,m,\eta}|^2 \, dx \le \int_{B_r(x)} |\nabla w|^2 \, dx + \frac{1}{\eta}|B_r(x)|
\end{equation}
for any $w\in  H^1(D)$, with $w-u_{i, m, \eta} \in  H_0^1(D)$.
Hence, the local Lipschitzianity of $u_{i, m, \eta}$  follows from  \cite[Theorem 3.3]{BMV_Lip}.

Furthermore, by the Poincare inequality, there exists a positive constant $C>0$, depending only on $D$, such that
\begin{equation}
\int_{D} |u_{i,m,\eta}|^2 \, dx \le C[  \norm{u_{i,m,\eta}}_{H^1(D)}+\norm{h}_{H^1(D)}].
\end{equation}
Then the estimate   \eqref{ineq_estimate_lip_u_imeta}  can be proven  carefully  keeping track of the constants involved in the proof of \cite[Theorem 3.3]{BMV_Lip}, see  \cite[Appendix A]{BMV_Lip}.
\end{proof}

\begin{proposition}{(Non-degeneracy)}\label{prop_non_dege}
There exists a dimensional constant $\kappa_d>0$ such that for any $x \in \Omega_{U_{m,\eta}}$ and any $r \in(0,d(x,\partial D))$,
\begin{equation}\label{ineq_non_dege}
\norm{U_{m,\eta}}_{L^\infty(B_r(x))} \ge \kappa_d \eta r.
\end{equation}
\end{proposition}
\begin{proof}
Inequality  \eqref{ineq_non_dege} follows from carefully keeping track of  the constants involved in the proof of \cite[Lemma 2.6]{MTV_reg_spec} and Remark \ref{remark_ineq_Omega_V}.
\end{proof}

\begin{corollary}{(Lower Density Estimates)}\label{corollary_density_estimates}
There exits a dimensional  constants $\delta_d>0$ such that
\begin{equation}\label{ineq_density}
|\Omega_U \cap B_r(x)| \ge \delta_d (\min\{1,d(x,\partial D)^{d-1}\})^2 \eta^2|B_r(x)|
\end{equation}
for any $x \in \partial  \Omega_U$ and any $r \in(0,d(x,\partial D))$.
\end{corollary}
\begin{proof}
In view of \eqref{ineq_estimate_lip_u_imeta} and \eqref{ineq_non_dege} the proof is standard, see for example \cite[Lemma 5.1]{V_book_free_boun}.
\end{proof}

\begin{corollary}{(Finiteness of the perimeter)}\label{corollary_finitness_pere}
There exits a dimensional  constants $k_d>0$ such that
\begin{equation}\label{ineq_peri}
\mathop{\rm Per}(\partial \Omega_U \cap B_r(x)) \le 
\begin{cases}
\eta k_d (d(x,\partial D)^{-2}+ \norm{\nabla |U|}_{L^\infty( B_r(x))}), & \text{ if } |\Omega_{U_{m,\eta}}| \le m,\\
\frac{k_d}{\eta} (d(x,\partial D)^{-2}+ \norm{\nabla |U|}_{L^\infty( B_r(x))}), &\text{ if } |\Omega_{U_{m,\eta}}| >m
\end{cases}
\end{equation}
for any $x \in \partial  \Omega_U$ and any $r \in(0,d(x,\partial D))$. In particular,
\begin{equation}
\mc{H}^{d-1}(\partial \Omega_U \cap K) <+\infty \text{ for any $K \subset D$ compact}.
\end{equation}
\end{corollary}
\begin{proof}
Thanks to Remark \ref{remark_ineq_Omega_V} and \cite[Lemma 2.5]{MTV_reg_vect}, it is enough  to keep track of the constants involved in \cite[Proof of claim 2.4 in Section 2B]{MTV_reg_vect}.
\end{proof}

\section{Shape variation of the free boundary and consequences}\label{sec_shape_var}
In this section we prove a shape variation formula which will allow us to show that for $\eta$ small enough, depending on $m$, $|\Omega_{m,\eta}|=m$ so that $U_{m,\eta}$ actually solves \eqref{prob_min_measure}.

\subsection{Shape Variation}\label{subsec_shape variation}

Letting $\mc{F}_0$ be the Dirichlet energy 
\begin{equation}\label{def_F0}
  \mc{F}_0 (W):=  \int_{D} |\nabla W|^2 \, dx \quad \text{ for any } W\in H^1(D,\R^k)
\end{equation}
then  the shape variation  $\delta \mc{F}_0$ at $W=(w_1,\dots, w_k)$ is defined as
\begin{equation}\label{def_deltaF0}
\delta\mc{F}_0(W) [\xi]:= \sum_{i=1}^k \int_{D} [-2\nabla  w_i\cdot D\xi   \nabla  w_i+|\nabla  w_i|^2 \dive(\xi)] \, dx
\quad \text{ for any } \xi \in C^1_c(D,\R^d).
\end{equation}

In view of Remark \ref{remark_ineq_Omega_V}, we may follow \cite[Proposition 11.2]{V_book_free_boun} to obtain the following shape variation formula.
\begin{proposition}\label{prop_shape_for}
The minimizer $U_{m,\eta}$ satisfies the shape variation equation 
\begin{equation}\label{eq_shape_derivative}
\delta\mc{F}_0 (U_{m,\eta})[\xi]= -\Lambda_{m,\eta} \int_{\Omega_{U_{m,\eta}}} \dive(\xi) \, dx  \quad  \text{ for any } \xi \in C^1_c(D,\R^d).
\end{equation}
for some constant $\Lambda_{m,\eta} \ge0$.
\end{proposition}

The next proposition, combined with a contradiction argument,  is a crucial tool for showing that $|\Omega_{m,\eta}|=m$ in the next section. 
\begin{proposition}\label{prop_unique_conti}
Let $U \in H^1(D,\R^k)$ be harmonic on $\Omega_U$. Suppose that 
\begin{equation}
\delta\mc{F}_0 (U)=0.
\end{equation}
Then $|D\setminus \Omega_U|=0$.
\end{proposition}

\begin{proof}
We follow the classical arguments exposed in \cite[Proposition 11.4]{V_book_free_boun} based on Almgren Monotonicity Formula.
We are going to prove a lower density bound for $\Omega_U$ for any $x \in D$. This implies $|D\setminus \Omega_U|=0$.
Since the argument is local and does not depend on the point $x \in D$, to simplify the notation we may assume $x=0$. 

Let us define the energy function and height function respectively as 
\begin{equation}\label{def_DH}
D(r):= \int_{B_r} |\nabla U |^2  \, dx, \quad \text{ and } \quad H(r) =  \int_{B_r} |U|^2  \, dx,
\end{equation}
for any $r \in (0, d(0,\partial D))$. 
Let us define, for any $r$ such that $ H(r) \neq 0$,
\begin{equation}\label{def_N}
\mc{N}(r):=\frac{r D(r)}{H(r)}.
\end{equation}
We make the following claim:
\begin{equation}\label{proof_unique_conti_1}
\text{ if } H>0 \text{ in }  (a,b) \subset (0, d(0,\partial D)) \text{ then  }  \mc{N}'(r) \ge 0 \text{ in } (a,b).
\end{equation}
Let us compute the derivative of $H$ and $D$.
By a change of variables and an integration by parts
\begin{equation}
H'(r)=\frac{d-1}{r}H(r)+2\sum_{i=1}^k\int_{\partial B_r} u_i \pd{u_i}{\nu}\, d\mc{H}^{d-1}=\frac{d-1}{r}H(r)+2D(r).
\end{equation}
On the other hand, by the Coarea formula
\begin{equation}
D'(r)= \int_{\partial B_r} |\nabla U|^2  \, dx.
\end{equation}
Testing for any $\delta \in (0,1)$ the condition $\delta\mc{F}_0 (U)[\xi]=0$ with the vector field $\xi_\delta:=x \phi_\delta(x)$, where $\phi_\delta\in C^\infty_c(B_r)$ is a radial cut off function with $\phi_\delta\equiv 1$ in $B_{(1-\delta)r}$ and $\nabla \phi_\delta=-\frac{1}{r\delta}\frac{x}{|x|}+o(\delta)$ as $\delta \to 0^+$ in  $B_r\setminus B_{(1-\delta)r}$,
we can argue as in \cite[Lemma 9,8]{V_book_free_boun} to show that 
\begin{equation}
-(d-2)D(r)+rD'(r)=2r\sum_{i=1}^k\int_{\partial B_r} \left|\pd{u_i}{\nu}\right|^2\, d\mc{H}^{d-1}.
\end{equation}
Hence,
\begin{multline}
\mc{N}'(r)= \frac{D(r)H(r)+rD'(r)H(r)-rD(r)H'(r)}{H^2(r)}\\
=\frac{D(r)H(r)+rD'(r)H(r)-rD(r)\left(\frac{d-1}{r}H(r)+2D(r)\right)}{H^2(r)}\\
=\frac{-(d-2)D(r)H(r)+rD'(r)H(r)-2rD^2(r)}{H^2(r)}\\
=\frac{2r}{H^2(r)}\left(H(r)\sum_{i=1}^k\int_{\partial B_r} \left|\pd{u_i}{\nu}\right|^2\, d\mc{H}^{d-1}-D^2(r)\right).
\end{multline}
Furthermore, since each $u_i$ us harmonic on $\Omega_U$,
\begin{equation}
D(r)= \int_{B_r} |\nabla U|^2  \, dx=\sum_{i=1}^k \int_{\partial B_r} u_i \pd{u_i }{\nu} \, d\mc{H}^{d-1},
\end{equation}
and so 
\begin{multline}
H(r)\sum_{i=1}^k\int_{\partial B_r} \left|\pd{u_i}{\nu}\right|^2\, d\mc{H}^{d-1}-D^2(r)\\
=\sum_{i,j=1}^k\left[\int_{\partial B_r} u_i^2\, d\mc{H}^{d-1}\int_{\partial B_r} \left|\pd{u_j}{\nu}\right|^2\, d\mc{H}^{d-1}
-\int_{\partial B_r} u_i \pd{u_i }{\nu} \, d\mc{H}^{d-1}\int_{\partial B_r} u_j\pd{u_j }{\nu} \, d\mc{H}^{d-1}\right]
\end{multline}
For any $i,j=1,\dots k$ the Cauchy-Schwartz inequality yields 
\begin{multline}
 \int_{\partial B_r} u_i^2\, d\mc{H}^{d-1}\int_{\partial B_r} \left|\pd{u_j}{\nu}\right|^2\, d\mc{H}^{d-1}+   
 \int_{\partial B_r} u_j^2\, d\mc{H}^{d-1}\int_{\partial B_r} \left|\pd{u_i}{\nu}\right|^2\, d\mc{H}^{d-1}\\
-2\int_{\partial B_r} u_i \pd{u_i }{\nu} \, d\mc{H}^{d-1}\int_{\partial B_r} u_j\pd{u_j }{\nu} \, d\mc{H}^{d-1} \\
\ge  \int_{\partial B_r} u_i^2\, d\mc{H}^{d-1}\int_{\partial B_r} \left|\pd{u_j}{\nu}\right|^2\, d\mc{H}^{d-1}+   
 \int_{\partial B_r} u_j^2\, d\mc{H}^{d-1}\int_{\partial B_r} \left|\pd{u_i}{\nu}\right|^2\, d\mc{H}^{d-1}\\
-2\left(\int_{\partial B_r} |u_i|  \, d\mc{H}^{d-1}+\int_{\partial B_r} \left|\pd{u_i }{\nu}\right|^2 \, d\mc{H}^{d-1}
\int_{\partial B_r} |u_j|^2\, d\mc{H}^{d-1}\int_{\partial B_r} \left|\pd{u_j}{\nu}\right|^2 \, d\mc{H}^{d-1} \right)^{\frac{1}{2}} \ge 0.
\end{multline}
The last inequality  is easy to check moving the last term to the right hand side and elevating to the square. 
In conclusion, $\mc{N}'(r)\ge 0$ and so we have proved \eqref{proof_unique_conti_1}.

It is then a standard computation, see for example the proof of \cite[Proposition 11.4]{V_book_free_boun}, to show that, letting  $r_0>0$ be such that $B_{r_0} \subset D$, the following doubling inequality holds:
\begin{equation}
\int_{B_{2r}} |U|^2 \, dx \le 2^{d-1} 4^{\mc{N}(r_0)} \int_{B_r} |U|^2 \, dx \quad  \text{ for any } r \in (0,r_0/2).
\end{equation}
 In the same proof also the following Cacciopoli type inequality is established (in our case we just need to sum from $i=1$ to  $i=k$) 
\begin{equation}
\int_{B_r} |\nabla U|^2 \, dx \le \frac{4}{r^2}\int_{B_{2r}} |U|^2 \, dx \quad  \text{ for any } r \in (0,r_0/2).
\end{equation}
Furthermore, by  \cite[Lemma 11.6]{V_book_free_boun}, there is a dimensional constant $C_d>0$ such that 
\begin{equation}
\int_{B_r}|U|^2\le C_d r^2 \left(\frac{|\Omega_U \cap B_r|}{|B_r|}\right)^{\frac{2}{d}} \int_{B_r}|\nabla |U||^2 
\le C_d r^2 \left(\frac{|\Omega_U \cap B_r|}{|B_r|}\right)^{\frac{2}{d}} \int_{B_r}|\nabla U|^2 
\end{equation}
if $\frac{|\Omega_U \cap B_r|}{|B_r|} \le \frac{1}{2}$.
Hence, 
\begin{equation}
\int_{B_r} |U|^2 \, dx \le C_d  \left(\frac{|\Omega_U \cap B_r|}{|B_r|}\right)^{\frac{2}{d}} 4^{\mc{N}(r_0)}\int_{B_r} |U|^2 \, dx. 
\end{equation}
It follow that, for some dimensional constant $C_d>0$,
\begin{equation}
\frac{|\Omega_U \cap B_r|}{|B_r|} \ge \min\left\{\frac{1}{2},\frac{1}{C_d2^{\mc{N}(r_0)d}}\right\}  \quad \text{ for any } r \in (0,r_0/2).
\end{equation}
In conclusion, we have proved the claimed lower density bound for $\Omega_U$ in any point $x \in D$.
\end{proof}

\subsection{ Equivalence with the constrained problem}\label{subsec_eqauivalence_const} In this subsection we show that for $\eta$ small enough $|\Omega_{m,\eta}|=m$. To this end, the key results are uniform (with respect to $\eta$) upper and lower bounds on $\Lambda_{m,\eta}$. Since we are going to argue by contradiction we need the following proposition.

\begin{proposition}\label{prop_limit_U_meta_H1}
Let $\{U_{m,\eta}\}$ be a family of minimizers of \eqref{prob_min_J_meta}. Any accumulation point $U_{m,0}$ of    $\{U_{m,\eta}\}$ in $H^1(D, \R^k)$  in weak sense  as $\eta \to 0^+$   is a solution of \eqref{prob_min_measure} and there exists a sequence $\{U_{m, \eta_n}\}$ such that 
\begin{equation}\label{eq_U_metan}
U_{m, \eta_n} \to U_{m,0} \text{ strongly in } H^1(D, \R^k) \quad \text{ and } \quad \chi_{U_{m,\eta_n}} \to \chi_{U_{m,0}} \text{ strongly in } L^1(D),
\end{equation} 
as $n \to \infty$.
\end{proposition}
\begin{proof}
By \eqref{ineq_U_meta_H1_bounded},  $\{U\}_{m,\eta}$ is bounded in $H^1(D)$. Let $\eta_n\to 0^+$ as $n\to \infty$ and let  $U_{m,0} \in H^1(D,\R^k)$  be such that
$U_{m,\eta_n} \rightharpoonup U_0$ weakly in $H^1(D,\R^k)$. It is not restrictive to suppose that, up to passing to a  subsequence, $U_{m, \eta_n} \to U_{m,0}$ pointwise in $D$. Then,
\begin{equation}
|\Omega_{U_{m,0}}|\le \liminf_{n\to \infty} |\Omega_{U_{m,\eta_n}}|.
\end{equation}
Furthermore, since $\{U\}_{m,\eta}$ is bounded in $H^1(D,\R^k)$, by Lemma \ref{lemma_estimtes_c_eta}, it must be
\begin{equation}
\liminf_{n\to \infty} |\Omega_{U_{m,\eta_n}}| \le m.
\end{equation}
By weak convergence,
\begin{equation}
\int_{D}|\nabla U_{m,0}|^2  \, dx \le \liminf_{n \to \infty}\int_{D}|\nabla U_{m,\eta_n}|^2  \, dx
\end{equation}
and choosing $U_{m,0}$ as a comparator in \eqref{prob_min_J_meta}, 
\begin{equation}
\liminf_{n \to \infty} \int_{D}|\nabla U_{m,\eta_n}|^2  \, dx 
\le \liminf_{n \to \infty}\left[\int_{D}|\nabla U_{m,0}|^2  \, dx +\eta_n |\Omega_{U_{m,0}}|\right]=\liminf_{n \to \infty}\int_{D}|\nabla U_{m,0}|^2
\end{equation}
so that  the convergence of $U_{m,\eta_n}$ to  $U_{m,0}$  is actually strong in  $H^1(D,\R^k)$. Furthermore, the same choice of comparator also implies
\begin{equation}
\liminf_{n \to \infty}\eta_n  (|\Omega_{U_{m,\eta_n}}|-m) \le \liminf_{n \to \infty} f(|\Omega_{U_{m,\eta_n}}|)
\le \liminf_{n \to \infty}  \eta_n  (|\Omega_{U_{m,0}}|-m)
\end{equation}
and so $\liminf_{n \to \infty}\Omega_{U_{m,\eta_n}}=|\Omega_{U_{m,0}}|$ which implies 
$\chi_{U_{m,\eta_n}} \to \chi_{U_{m,0}}$ strongly in $L^1(D)$ as $n \to \infty$.
Finally, for any comparator  $W$ of  \eqref{prob_min_J_meta} with $|\Omega_W| \le m$
\begin{equation}
\int_{D}|\nabla U_{m,0}|^2  \, dx=\liminf_{n \to \infty} \int_{D}|\nabla U_{m,\eta_n}|^2  \, dx 
\le\liminf_{n \to \infty} \left[ \int_{D}|\nabla W|^2  \, dx +\eta_n |\Omega_W|\right]\\
= \int_{D}|\nabla W|^2  \, dx.
\end{equation}
Hence, $U_{m,0}$ is a minimizer of \eqref{prob_min_measure_relaxed} and by Proposition \ref{prop_minim} also of \eqref{prob_min_measure}.
\end{proof}

Thanks to the previous proposition and Subsection \ref{subsec_shape variation} we are now in position to prove uniform   bounds on $\Lambda_{m,\eta}$.
\begin{proposition}
There exist positive constants $\eta_m, c_m, C_m$, depending on $m$, such that 
\begin{equation}\label{ineq_Lambda_meta_estimates}
c_m \le \Lambda_{m,\eta} \le C_m \quad \text{ for any } \eta \in (0,\eta_m].
\end{equation}
\end{proposition}

\begin{proof}
Let $\eta_n \to 0^+$ be any sequence and $\xi \in C^{\infty}_c(D,\R^d)$.  Then, letting  $U_{m,0}$ be as in Proposition \ref{prop_limit_U_meta_H1} and  up to passing to a  subsequence, \eqref{eq_U_metan} holds. Passing to the limit in \eqref{eq_shape_derivative} we obtain
\begin{equation}
\delta\mc{F}(U_{m,0}) [-\xi]=\left(\lim_{n \to \infty}\Lambda_{m,\eta_n} \right)\int_{\Omega_{U_{m,0}}} \dive(\xi)\, dx.
\end{equation}
By Proposition \ref{prop_unique_conti}, since $|\Omega_{U_{m,0}}|=m$,  
\begin{equation}
\lim_{n \to \infty}\Lambda_{m,\eta_n}>0.
\end{equation}
Furthermore, letting $\xi_0 \in C^\infty_c(\Omega,\R^d)$ be such that 
\begin{equation}
\int_{\Omega_{U_{m,0}}} \dive(\xi_0) \, dx= 1
\end{equation}
(the  existence of such a field is guarantied for example by \cite[Lemma 11.3]{V_book_free_boun}),
we conclude that 
\begin{equation}
\lim_{n \to \infty}\Lambda_{m,\eta_n}\le\delta\mc{F}(U_{m,0}) [-\xi_0].
\end{equation}
Hence, we have proved \eqref{ineq_Lambda_meta_estimates}.
\end{proof}

\begin{proposition}\label{prop_La_meta_unifor}
There exists $\widetilde{\eta}_m$ such that 
\begin{equation}\label{eq_measure_Omega_Umeta}
|\Omega_{U_{m,\eta}}|=m, \quad \text{for any } \eta \in (0,\widetilde{\eta}_m]
\end{equation}
for any minimizer $U_{m,\eta}$ of \eqref{prob_min_J_meta}.
\end{proposition}

\begin{proof}
We argue by contradiction supposing at first that  $|\Omega_{U_{m,\eta}}|> m$. Let $x_0 \in \partial \Omega_{U_{m,\eta}}$ be such that $|B_r(x_0)\cap \Omega_{U_{m,\eta}} |< |B_r(x_0)|$ with  $B_r(x_0) \subset D$ and $|B_r(x_0)|<|\Omega_{U_{m,\eta}}|-m$.
Let us consider a family of vector fields $\Phi_t:=Id-t\xi$ with $t>0$ and  $\xi \in C^1_c(B_r(x_0))$ such that 
\begin{equation}
\int_{\Omega_{U_{m,\eta}}} \dive(\xi)\, dx=\int_{\Omega_{U_{m,\eta}}\cap B_r(x_0)} \dive(\xi)\, dx=1.   
\end{equation}
The existence of such a vector field is granted by \cite[Lemma 11.3]{V_book_free_boun} applied to the open set  $B_r(x_0) \cap \Omega_{U_{m,\eta}}$in the domain $B_r(x_0)$ since $|B_r(x_0)\cap \Omega_{U_{m,\eta}} |< |B_r(x_0)|$. 

For $t>0$ small enough $\Phi_t$ is a diffeomorphism 
and, letting  $U_{m,\eta_n,t}:=U_{m,\eta_n} \circ \Phi^{-1}_t$, 
\begin{equation}
|\Omega_{U_{m,\eta,t}}|=|\Omega_{U_{m,\eta}}| -t +o(t), 
\end{equation}
see for example \cite[Lemma 9.5]{V_book_free_boun}. Standard shape variation computations, that is, using $U_{m,\eta_n,t}$ as a comparator, dividing by $t$  and then  passing to the limit as $t\to 0^+$, yield
\begin{equation}
-\delta \mc{F}_0(U_{m,\eta} )[\xi] -\frac{1}{\eta} \ge 0,
\end{equation}
see   \cite[Lemma 9.5]{V_book_free_boun} for the details. Thus, by \eqref{eq_shape_derivative}, and  \eqref{ineq_Lambda_meta_estimates}, if $\eta\le \eta_m$
\begin{equation}
C_m \ge  \Lambda_{m,\eta} \ge \frac{1}{\eta}.
\end{equation}
Hence, $|\Omega_{U_{m,\eta}}|\le m$ for any $\eta < \min\{\eta_m, 1/C_m\}$.

Similarly, if  $|\Omega_{U_{m,\eta}}|<m$,
we may consider   $\Phi_t:=Id+t\xi$ with $t>0$ and $\xi$ as above to show that,  by  \eqref{ineq_Lambda_meta_estimates},  if $\eta\le \eta_m$
\begin{equation}
c_m \le  \Lambda_{m,\eta} \le \eta.
\end{equation}
In conclusion,   $|\Omega_{U_{m,\eta}}|= m$ for any $\eta < \min\{\eta_m,1/ C_m,c_m\}$.
\end{proof}

We are now in position to prove Theorem \ref{theo_reg_m_fixed}.
\begin{proof}[\textbf{Proof of Theorem \ref{theo_reg_m_fixed}.}]
It is possible to show that for any $x_0\in \text{\rm{Reg}}(\partial \Omega_{U_{m,\eta}}) \cup  \text{\rm{Sing}}_1(\partial \Omega_{U_{m,\eta}})$, $U_{m,\eta}$ is a viscosity solution   of 
\begin{equation}\label{prob_Ue_vis}
 \begin{cases}
- \Delta U_{m,\eta} = 0, & \text{in } \Omega_{U_{m,\eta}} \cap B_r(x_0), \\
U_{m,\eta} = 0, & \text{on } \partial \Omega_{U_{m,\eta}}  \cap B_r(x_0), \\
|\nabla |U_{m,\eta}|| = \sqrt{\Lambda_{m,\eta}}, & \text{on } \partial \Omega_{U_{m,\eta}} \cap B_r(x_0),
\end{cases}
\end{equation}
for some small $r>0$, arguing as in \cite[Lemma 3.2]{MTV_reg_vect}. Hence the regularity stated in  Theorem \ref{theo_reg_m_fixed} can be proved for $U_{m,\eta}$ with the improvement of flatness result  \cite{DT_impr} or following  \cite{MTV_reg_vect} (see also \cite{MTV_reg_spec})   with no modifications.  For this second approach we notice that it is also possible to prove the monotonicity of the Weiss formula (see \cite[Section 2D]{MTV_reg_vect}),  for $U_{m,\eta}$ in view of \eqref{eq_shape_derivative} and  the  equipartition of energy for solutions of \eqref{eq_shape_derivative} which can be obtained as in \cite[Proposition 9.8]{V_book_free_boun}. Finally, by  Proposition \ref{prop_La_meta_unifor},  $U_{m,\eta}$ is a solution of \eqref{prob_min_measure} for any $\eta \in (0,\widetilde \eta_m]$ and so we have proved   Theorem \ref{theo_reg_m_fixed}.
\end{proof}

\begin{remark}
We have actually only shown  that there exists a solution of \eqref{prob_min_measure} with the regularity stated in Theorem \ref{theo_reg_m_fixed}, not that Theorem \ref{theo_reg_m_fixed} holds for  any solution of  \eqref{prob_min_measure}.  However, we may repeat our analysis for the functional
\begin{equation}\label{def_J_etaU}
J_{m,\eta,U}(W):H^1(D) \to [0,+\infty), \quad J_{m,\eta,U}(W):=\int_{D} |\nabla W|^2+W^2-2WU\, dx + f_{m,\eta} (|\Omega_W|),
\end{equation}
with standard modifications. Indeed, $U \in L_{loc}^\infty(D)$, an so we can show that any minimizer is Lipschitz and not degenerate, and that  $|U_{m,\eta,U}|=m$  for $\eta>0$ small enough, which implies  it must be $U_{m,\eta,U}=U$. 
\end{remark}

\section{Asymptotics with respect to the measure constraint}\label{sec_uni}
In this section we study the asymptotics as $m \to |D|^-$ of minimizers $U_m$ of \eqref{prob_min_measure} and we use, for the sake of simplicity, the parameter  
$\e:=|D|-m$. In other words we study the asymptotics of   $U_\e$, minimizer of  \eqref{prob_min_measure} with $m=|D|-\e$, as $\e \to 0^+$.

Let $h\in H^1(D,\R^k)$ be the harmonic extension of the boundary datum $g$ in \eqref{prob_min_measure}, that is,  the unique solution of 
\begin{equation}
\begin{cases}
-\Delta h=0, &\text{ in } D,\\
h=g, &\text{ on }  \partial D.   
\end{cases}
\end{equation}
Suppose that $h$ has a unique   zero in $D$ and that it has  order $1$ while its  Jacobian matrix has trivial kernel, that is, there exists $x_0 \in D$ such that 
\begin{equation}
h(x_0)=0, \quad h(x)\neq 0 \text{ if } x\neq x_0\quad \text{ and } \quad   \mathop{\rm{Ker}}(\nabla h(x_0))=\{0\},
\end{equation}
where $\mathop{\rm{Ker}}(\nabla h(x_0))$ is the kernel of the linear map  $\nabla h(x_0): \R^d \to \R^k$.
Without loss of generality we may suppose that $0 \in D$ and that $x_0=0$ to simplify the notations.

Let us define  $V_\e:=h-U_\e$. Since $h$ is harmonic,
\begin{equation}\label{eq_Ve_Ue}
\int_{D} |\nabla V_\e|^2\, dx =\int_D |\nabla U_\e|^2\, dx -\int_D |\nabla h|^2\, dx
\end{equation}
 and $V_\e$ is a minimizer of
\begin{equation}\label{prob_min_Ve}
\inf\left\{\int_{D} |\nabla W|^2\, dx:W\in  H_0^1(D, \R^k),  |\{W(x)=h(x)\}|=\e\right\}.
\end{equation}
Letting $D_\e:=\e^{-\frac{1}{d}}D$,
\begin{equation}\label{def_Ve_he}
\widetilde V_\e(x):=\e^{-\frac{1}{d}}V(\e^{\frac{1}{d}} x) \quad \text{ and } \quad  \widetilde h_\e(x):=\e^{-\frac{1}{d}}h(\e^{\frac{1}{d}} x)  \quad \text{ for any  } x \in  D_\e,
\end{equation}
the rescaled function $\widetilde V_\e(x)$  minimizes  
\begin{equation}\label{prob_min_tildeVe}
\inf\left\{\int_{D_\e} |\nabla W|^2\, dx:W\in  H_0^1(D_\e, \R^k),  |\{W(x)=\widetilde h_\e(x)\}|=1\right\}
\end{equation}
and 
\begin{equation}
\frac{1}{\e}\int_D |\nabla V_\e|^2 \, dx =\int_{D_\e} |\nabla \widetilde V_\e|^2 \, dx.
\end{equation}
Furthermore, let us recall the definition of the  homogeneous Sobolev space
\begin{equation}
D^{1,2}(\R^d, \R^k):=\{W \in H^1_{loc}(\R^d,\R^k): \nabla W \in L^2(\R^d,\R^{d,k})\},
\end{equation}
 endowed with the norm
\begin{equation}
\norm{W}_{D^{1,2}(\R^d, \R^k)}:=\left(\int_{\R^d}|\nabla W|^2 \, dx\right)^{\frac{1}{2}} +\left(\int_{B_1}|W|^2 \, dx\right)^{\frac{1}{2}}
\end{equation}
as a functional setting for the minimization problem 
\begin{equation}\label{prob_min_tildeV0}
\inf\left\{\int_{\R^d} |\nabla W|^2\, dx:W\in D^{1,2}(\R^d, \R^k),  |\{W(x)= \nabla h(0)x\}|=1\right\}.
\end{equation}

For any  constant matrix $A \in \R^{k,d}$, we can show that the minimization in \eqref{prob_min_tildeV0} is equivalent to minimize only over functions with compact support or with just an upper bound on the measure of the contact set  $\{W(x)=  Ax\}$.
\begin{proposition}\label{prop_equiv_ge}
Let $A$ be a $k\times d$ matrix. Then 
\begin{multline}\label{eq_equiv_ge}
\inf\left\{\int_{\R^d} |\nabla W|^2\, dx:W\in D^{1,2}(\R^d, \R^k),  |\{W(x)= Ax\}|=1\right\}\\
=\inf\left\{\int_{\R^d} |\nabla W|^2\, dx:W\in D^{1,2}(\R^d, \R^k),  |\{W(x)= Ax\}|\ge1\right\}\\
=\inf\left\{\int_{\R^d} |\nabla W|^2\, dx:W\in D^{1,2}(\R^d, \R^k),  |\{W(x)= Ax\}|\ge1, \mathop{\rm supp}(W) \text{ is compact}\right\}\\
=\inf\left\{\int_{\R^d} |\nabla W|^2\, dx:W\in D^{1,2}(\R^d, \R^k),  |\{W(x)= Ax\}|=1, \mathop{\rm supp}(W) \text{ is compact}\right\}.
\end{multline}
\end{proposition}

\begin{proof}
The scaling $W_\delta:=\frac{1}{\delta^\frac{1}{d}}W(\delta^\frac{1}{d}x)$ for any $\delta \ge 1$ yields 
\begin{multline}
\inf\left\{\int_{\R^d} |\nabla W|^2\, dx:W\in D^{1,2}(\R^d, \R^k),  |\{W(x)= Ax\}|\ge1\right\}\\
=\inf_{\delta\ge 1}\inf\left\{\int_{\R^d} |\nabla W|^2\, dx:W\in D^{1,2}(\R^d, \R^k),  |\{W(x)= Ax\}|=\delta\right\}\\
=\inf_{\delta\ge 1}\inf\left\{\delta\int_{\R^d} |\nabla W|^2\, dx:W\in D^{1,2}(\R^d, \R^k),  |\{W(x)= Ax\}|=1\right\}\\
=\inf\left\{\int_{\R^d} |\nabla W|^2\, dx:W\in D^{1,2}(\R^d, \R^k),  |\{W(x)= Ax\}|=1\right\}.
\end{multline}
To prove the second equality in \eqref{eq_equiv_ge}, it is enough to show that 
\begin{multline}
\inf\left\{\int_{\R^d} |\nabla W|^2\, dx:W\in D^{1,2}(\R^d, \R^k),  |\{W(x)= Ax\}|=1\right\}\\
\le \inf\left\{\int_{\R^d} |\nabla W|^2\, dx:W\in D^{1,2}(\R^d, \R^k),  |\{W(x)= Ax\}|\ge1, \mathop{\rm supp}(W) \text{ is compact}\right\},
\end{multline}
since the reverse inequality is obvious. Let $W$ be a comparator for  the first problem in \eqref{eq_equiv_ge}.
If $d \ge 3$ let us consider  cut off functions
$\eta_R \in C^\infty_c(\R^d)$ such that $\eta_R  =1$ in $B_R$, $\eta_R=0$ on $\R^{d}\setminus B_{2R}$ and $|\nabla \eta|\le \frac{2}{R}$. If  $d =2 $ instead we
define
\begin{equation}
\eta_R(x):=
\begin{cases}
1, &\text{ if } x \in B_R,\\
\frac{\log(R^2)-\log(|x|)}{\log(R^2)-\log(R)}, &\text{ if }x \in B_{R^2}\setminus B_R,\\
0, &\text{ if }x \in \R^d\setminus B_{R^2}.
\end{cases}
\end{equation}
Letting $\delta_R:=|\{W(x)= Ax\} \cap B_R|$ and  choosing $R$ large enough so that $ \delta_R \neq 0$,
\begin{equation}
W_R(x):= \frac{1}{\delta_R^{1/d}}\eta_R( \delta_R^{{1/d}} x)W(\delta_R x),
\end{equation}
is a comparator for the third minimization problem in  \eqref{eq_equiv_ge}. A change of variables yields
\begin{multline}\label{proof_prop_equiv_ge_1}
\int_{\R^d}  |\nabla W_R |^2 \, dx \le \delta_R \int_{\R^d}  |\nabla W |^2 \eta_R^2\, dx\\
+2\delta_R\sum_{i=1}^{k}\int_{\R^d} \eta_R w_i \nabla w_i\cdot \nabla \eta_R \, dx 
+\delta_R \int_{\R^d}  | W |^2 |\nabla \eta_R|^2\, dx.
\end{multline}
Furthermore, if $d \ge 3$, 
\begin{equation}
\int_{\R^d}  | W |^2 |\nabla \eta_R|^2\, dx \le 2\int_{\R^d\setminus B_{R}}  \frac{| W |^2}{|x|^2}\, dx 
\end{equation}
 and, by the Hardy inequality, 
\begin{equation}
\int_{\R^d}  \frac{| W |^2}{|x|^2}\, dx <+\infty
\end{equation}
for any $W \in D^{1,2}(\R^d,\R^k)$.
If $d =2$, 
\begin{equation}
\int_{\R^d}  | W |^2 |\nabla \eta_R|^2\, dx \le \int_{B_{R^2}\setminus B_R}  \frac{| W |^2}{|\log(R)|^2|x|^2}\, dx.
\end{equation}
Furthermore, for any $x \in B_{R^2}\setminus B_R$ and for some constant $C>0$ that does not depend on $R$
\begin{equation}
\frac{1}{|\log(R)|^2} \le   \frac{C}{1+|\log(x)|^2}.
\end{equation}
and we also have the following Hardy-type inequality 
\begin{equation}\label{ineq_Hardy_d=2}
 \int_{\R^2}  \frac{| W |^2}{(1+|\log(|x|)|^2)|x|^2}\, dx\le 4\int_{\R^2}  | \nabla W |^2 \, dx
\end{equation}
for any $W \in D^{1,2}(\R^d,\R^k)$.
In conclusion, since $\delta_R \to 1$ as $R \to +\infty$, passing to the limit as $R\to +\infty$ in \eqref{proof_prop_equiv_ge_1},
\begin{equation}
\lim_{R \to \infty}\int_{\R^d}  |\nabla W_R |^2 \, dx \le \int_{\R^d}  |\nabla W |^2\, dx.
\end{equation}
Hence, we have proved the third equality in \eqref{eq_equiv_ge}. Finally,
\begin{multline}
\inf\left\{\int_{\R^d} |\nabla W|^2\, dx:W\in D^{1,2}(\R^d, \R^k),  |\{W(x)= Ax\}|\ge1, \mathop{\rm supp}(W) \text{ is compact}\right\}\\
=\inf_{R>0}\inf\left\{\int_{\R^d} |\nabla W|^2\, dx:  |\{W(x)= Ax\}|\ge1, W \in H_0^1(B_R,\R^k)\right\}\\
=\inf_{R>0}\inf\left\{\int_{\R^d} |\nabla W|^2\, dx: |\{W(x)= Ax\}|=1, W \in H_0^1(B_R,\R^k)\right\}\\
=\inf\left\{\int_{\R^d} |\nabla W|^2\, dx:W\in D^{1,2}(\R^d, \R^k),  |\{W(x)= Ax\}|=1, \mathop{\rm supp}(W) \text{ is compact}\right\},
\end{multline}
where the third  equality is a consequence of Proposition \ref{prop_minim}. Indeed, for any $R>0$, by harmonicity of $Ax$, letting $U(x):=Ax-W(x)$
\begin{multline}
\inf\left\{\int_{\R^d} |\nabla W|^2\, dx: |\{W(x)= Ax\}|\ge1, W \in H_0^1(B_R,\R^k)\right\}\\
=\inf\left\{\int_{B_R} |\nabla U|^2\, dx: |\{U=0\}|\ge 1, U-A \in H_0^1(B_R,\R^k)\right\}-\int_{B_R} |\nabla Ax|^2\, dx\\
=\inf\left\{\int_{B_r} |\nabla U|^2\, dx: |\Omega_U|\le |B_R|-1, U-A \in H_0^1(B_R,\R^k)\right\}-\int_{B_R} |\nabla Ax|^2\, dx\\
=\inf\left\{\int_{B_r} |\nabla U|^2\, dx: |\Omega_U|= |B_R|-1, U-A \in H_0^1(B_R,\R^k)\right\}-\int_{B_R} |\nabla Ax|^2\, dx\\
=\inf\left\{\int_{\R^d} |\nabla W|^2\, dx: |\{W(x)= Ax\}|=1, W \in H_0^1(B_R,\R^k)\right\},
\end{multline}
as previously observed.
\end{proof}

We are going to show that, under assumption \eqref{hp_h}, the limit as $\e \to 0^+$ of \eqref{prob_min_tildeVe} is \eqref{prob_min_tildeV0} and, 
that \eqref{prob_min_tildeV0} admits of minimizers as limits of sequences $\{\widetilde V_{\e_j}\}_{j \in\mb{N}\setminus\{0\}}$ with $\e_j \to 0^+$ as $j \to\infty$. 
We need two technical lemmas.

\begin{lemma}\label{lemma_psij}
Let $f_j\in H_{loc}^1(\R^d,\R^k)$,  and suppose that $f_j \to f_\infty $ strongly in  $H_{loc}^1(\R^d,\R^k)$ for some $f_\infty \in H_{loc}^1(\R^d,\R^k)$. Let $R>0$, $K\subset B_R$  be a compact set with  $|K|=1$  and let $\psi_j \in H_0^1(B_R,\R^k)$ be the unique solution of  
\begin{equation}\label{prob_psij}
\begin{cases}
-\Delta \psi_j=0, &\text{ in } B_R\setminus K,\\
\psi_j=f_j, &\text{ on }  K, \\
\psi_j=0, &\text{ on }  \partial B_R,
\end{cases}
\end{equation}
for any $j \in \mb{N}$ or $j=\infty$. Then 
\begin{equation}\label{limit_psij}
\lim_{j \to \infty}\int_{B_R} |\nabla \psi_j|^2 \, dx=\int_{B_R} |\nabla \psi_\infty|^2 \, dx.
\end{equation}
\end{lemma}
\begin{proof}
Let $r_1,r_2>0$ be such that $K \subset B_{r_1} \subset B_{r_2} \subset B_R$. Let $\eta$ be cut off function such that $\eta=1$ in $B_{r_1}$ and $\eta=0$ in $B_R\setminus B_{r_2}$.   By harmonicity of $\psi_j-\psi_\infty$, it follows that
\begin{multline}
\int_{B_R} |\nabla \psi_j-\nabla \psi_\infty|^2 \, dx \le \int_{B_R} |\nabla [\eta (f_j-f_\infty)]|^2 \, dx\\
=\int_{B_R} [\eta^2 |\nabla (f_j-f_\infty)|^2+\eta (f_j-f_\infty) \nabla(f_j-f_\infty)\cdot \nabla \eta+| (f_j-f_\infty)|^2  |\nabla \eta|^2]  \, dx. 
\end{multline}
Since $\eta$ and $\nabla \eta$ are bounded, passing to the limit as $j\to\infty$ we have proved \eqref{limit_psij}.
\end{proof}

Let us set some notation that we are going to use in the next lemma. For any $R>0$ we  define 
\begin{equation}
D^{1,2}(\R^d\setminus B_R, \R^k):=\{W \in H^1(B_{R'}\setminus B_R,\R^k)\text{ for any } R'>0: \nabla W \in L^2(\R^d\setminus B_R,\R^{d,k})\},
\end{equation}
and similarly let
\begin{multline}
D^{1,2}_{0,\partial B_R}(\R^d\setminus B_R, \R^k):=\{W \in H^1(B_{R'}\setminus B_R,\R^k)\text{ for any } R'>0,\\
W=0 \text{ on } \partial B_R: \nabla W \in L^2(\R^d\setminus B_R,\R^{d,k})\}.
\end{multline}
The following lemma is the key result of the whole section and the main ingredient in the existence of minimizers of \eqref{prob_min_tildeV0}. \begin{lemma}\label{lemma_bound_capa}
We have that
\begin{equation}\label{ineq_capa_Ve}
\limsup_{\e \to 0^+}\int_{D_\e}|\nabla \widetilde V_\e|^2 \, dx<+\infty.
\end{equation}
Furthermore, if \eqref{hp_h} holds, there exists a constant $\kappa>0$, that does not depends on $R$, and $R_0>0$,  such that for any $R \ge R_0$  
\begin{equation}\label{limit_meausure_Ve}
\limsup_{\e \to 0^+}|\{\widetilde V_\e= \widetilde h_\e\}\setminus B_R| \le \frac{\kappa}{R^2}.
\end{equation}
\end{lemma}
\begin{proof}
Let $\eta \in C_c^\infty(\R^d)$ be a cut off function such that $\eta=1$ on 
$B_{r}$ with $r:=(1/\omega_d)^\frac{1}{d}$, $\eta <1$ in $B_{2r}\setminus B_r$ and $\eta=0$ in $\R^d\setminus B_{2r}$. Let $\e>0$ be  such that $B_{2r} \subset D_\e$. Testing \eqref{prob_min_tildeVe} with  
$\varphi:=\eta h_\e$ we obtain, letting $r_\e:=r\e^{\frac{1}{d}}$,
\begin{multline}
\int_{D_\e}|\nabla \widetilde V_\e|^2 \, dx \le \int_{B_{2r}}|\nabla (\eta h_\e)|^2 \, dx \le 
\norm{\eta}^2_{C^1(\R^d)}\int_{B_{2r}}[|h_\e|^2 +|\nabla  h_\e|^2] \, dx \\
\le  \norm{\eta}^2_{C^1(\R^d)}\e^{-1} \left(\e^{-\frac{2}{d}}\int_{B_{2r_\e}}|h|^2 \, dx+\int_{B_{2r_\e}}|\nabla h|^2 \, dx\right) 
\le C  \norm{\eta}^2_{C^1(\R^d)} \norm{\nabla h}^2_{L^\infty(B_{2r},\R^{k,d})},
\end{multline}
for some dimensional constant $C>0$, thanks to \eqref{hp_h}.

Next we prove  \eqref{limit_meausure_Ve}. Let $R>0$ and  $\e>0$ small enough so that $B_R \subset D_\e$. Let us consider the harmonic extension $\Phi_{\e,R}$ of $\widetilde V_\e$ from $\partial B_R$ in $\R^d\setminus B_R$, that is, 
\begin{equation}
\Phi_{R,\e}(x)=\frac{|x|^2-R^2}{d \omega_d R} \int_{\partial B_R} \frac{\widetilde V_\e(\xi)}{|\xi-x|^d} \, d\mc{H}^{d-1}_\xi \quad 
\text{ for any } x \in \R^d\setminus \overline{B_R}.
\end{equation}
Since $|\nabla W| \ge|\nabla |W||$  and $\Phi_{R,\e}$ is harmonic with Sobolev trace $\widetilde V_\e$ on $\partial B_R$, we have 
\begin{multline}\label{proof_lemma_bound_capa_1}
\int_{\R^d}|\nabla  \widetilde V_\e|^2 \, dx \ge \int_{\R^d\setminus B_R}|\nabla \widetilde V_\e|^2 \, dx\\
\ge \inf_{W \in D^{1,2}(\R^d\setminus B_R,\R^k)}\left\{\int_{\R^d\setminus B_R}|\nabla W|^2 \, dx
:W=\widetilde h_\e \text{ on } \{\widetilde V_\e=\widetilde h_\e\}\setminus B_R \text{ and } W=\widetilde V_\e \text{ on } \partial B_R \right\} \\
\ge \inf_{W \in D^{1,2}_{0,\partial B_R}(\R^d\setminus B_R,\R^k)}\left\{\int_{\R^d\setminus B_R}|\nabla W|^2 +|\nabla \Phi_{\e,R}|^2 \, dx
:W= \widetilde h_\e-\Phi_{R,\e} \text{ on } \{\widetilde V_\e=\widetilde h_\e\}\setminus B_R\right\} \\
\ge \inf_{w \in D^{1,2}_{0,\partial B_R}(\R^d\setminus B_R)}\left\{\int_{\R^d\setminus B_R}|\nabla w|^2\, dx
:w=|\widetilde h_\e-\Phi_{R,\e}| \text{ on } \{\widetilde V_\e=\widetilde h_\e\}\setminus B_R  \right\}\\
\ge \inf_{w \in D^{1,2}_{0,\partial B_R}(\R^d\setminus B_R)}\left\{\int_{\R^d\setminus B_R}|\nabla w|^2 \, dx
:w=|\widetilde h_\e-\Phi_{R,\e}| \text{ on } \{\widetilde V_\e=\widetilde h_\e\}\setminus B_{2R}  \right\}\\
\ge \inf_{x \in \R^d\setminus B_{2R}}(|\widetilde h_\e|-|\Phi_{R,\e}|)^2
\capa{}{\{\widetilde V_\e=\widetilde h_\e\}\setminus B_{2R}},
\end{multline}
where $\capa{}{\cdot}$ denotes the classical Sobolev capacity in $\R^d$.

Now  we estimate $|\tilde{h}_\e|$ from below  and  $|\Phi_{R,\e}|$ from above on $\R^d \setminus B_{2R}$.
Let $x_\e$ be a sequence of minimizer for $\inf_{x \in  \R^d\setminus B_{2R}}|\widetilde h_\e(x)|$. Since $h(x)\neq 0$ if $x \neq 0$, we must have  $\e^{\frac{1}{d}}|x_\e|\to 0^+$, otherwise $h_{\e}(x_\e)$ is not bounded.  A Taylor expansion  around $0$ yields
\begin{equation}
h_{\e}(x_\e)=\e^{-\frac{1}{d}}h(\e^{\frac{1}{d}}x_\e)=\nabla h(0)x_\e +o(|x_\e|), \quad \text{ as } \e \to 0^+.
\end{equation}
It follows that
\begin{equation}
\liminf_{\e \to 0^+}|h_{\e}(x_\e)| \ge |x_\e|\min_{x \in \partial B_1}|\nabla h(0)x| \quad \text{ as } \e \to 0^+
\end{equation}
thus $\{x_\e\}$ is bounded and, up to passing to a subsequence, $x_\e \to x_0$ for some $x_0 \in \R^d\setminus B_{2R}$.
Since $\widetilde h_\e$ are  equi-Lipschitz,
\begin{equation}\label{proof_lemma_bound_capa_2}
\lim_{\e \to 0^+ }|h_{\e}(x_\e)| = |\nabla h(0)x_0| \ge |x_0| \min_{x \in \partial B_1} |\nabla h(0)x|\ge 2R \min_{x \in \partial B_1} |\nabla h(0)x|.
\end{equation}
On the other hand for any $x \in \R^d \setminus {B_{2R}}$
\begin{equation}
|\Phi_{R,\e}(x)|\le \frac{|x|^2-R^2}{d \omega_d R} \int_{\partial B_R} \frac{|\widetilde V_\e(\xi)|}{|\xi-x|^d} \, d\mc{H}^{d-1}_\xi \le\frac{|x|^2-R^2}{d\omega_dR(|x|-R)^d}
\int_{\partial B_R}|\widetilde V_\e(\xi)|\, d\mc{H}^{d-1}_\xi.
\end{equation}
Furthermore, letting $t:=|x|/R$,
\begin{equation}
\frac{|x|^2-R^2}{R(|x|-R)^d}= \frac{t^2-1}{R^{d-1}(t-1)^d},
\end{equation}
thus, since  $t \to \frac{t^2-1}{(t-1)^d}$ is bounded in $[2,+\infty)$, it follow that there exists a dimensional constant $C>0$ such that 
\begin{equation}\label{proof_lemma_bound_capa_3}
|\Phi_{R,\e}(x)| \le\frac{C}{R^{d-1}}\int_{\partial B_R}|\widetilde V_\e(\xi)|\, d\mc{H}^{d-1}_\xi \quad \text{ for any }x \in \R^d \setminus {B_{2R}}.
\end{equation}
Let us show that for any $\delta>0$ there exists $R_0>0$ such that for any $R \ge R_0$ and any $\e>0$,
\begin{equation}\label{proof_lemma_bound_capa_4}
\frac{1}{R^{d-1}}\int_{\partial B_R}|\widetilde V_\e(\xi)|\, d\mc{H}^{d-1}_\xi \le \delta R.
\end{equation}
By the H\"older inequality,  the classical Sobolev trace inequality on $B_1$ and a scaling argument 
\begin{multline}
\frac{1}{R^d}\int_{\partial B_R}|\tilde{V}|\, d\mc{H}^{d-1} \le \frac{\sqrt{d \omega_d}}{R^{(d+1)/2}}
\left(\int_{\partial B_R}|\widetilde V_\e|\, d\mc{H}^{d-1}\right)^{\frac{1}{2}}     \\
\le \frac{C_2}{R^{(d+1)/2}}\left(\int_{B_R}\frac{1}{R}|\widetilde V_\e|^2+R|\nabla\widetilde V_\e|^2\, dx\right)^{\frac{1}{2}},
\end{multline}
for some dimensional constant $C_2>0$.
Furthermore, if $d \ge 3$ by the classical Hardy inequality 
\begin{equation}
\frac{1}{R^{d+2}}\int_{ B_R}|\widetilde V_\e|^2\, dx \le \frac{1}{R^{d}}\int_{ B_R}\frac{|\widetilde V_\e|^2}{|x|^2}\, dx
\le\frac{C_3}{R^{d}}\int_{ \R^d}|\nabla \widetilde V_\e|^2 \, dx,
\end{equation}
for some dimensional constant $C_3>0$. If $d=2$ by \eqref{ineq_Hardy_d=2}
\begin{equation}
\frac{1}{R^{4}}\int_{ B_R}|\widetilde V_\e|^2\, dx \le \frac{1+|\log(R)|^2}{R^2}\int_{ B_R}\frac{|\widetilde V_\e|^2}{(1+|\log(|x|)|^2)|x|^2}\, dx
\le \frac{1+|\log(R)|^2}{4R^2}\int_{ \R^d}|\nabla \widetilde V_\e|^2 \, dx.
\end{equation}
In view of \eqref{ineq_capa_Ve}, we have proved \eqref{proof_lemma_bound_capa_4} in any dimension $d\ge2$.

Putting \eqref{proof_lemma_bound_capa_2}, \eqref{proof_lemma_bound_capa_3} and \eqref{proof_lemma_bound_capa_4} together we have shown that, if $R$ is large enough,
\begin{equation}
\lim_{\e\to 0^+}|\widetilde h_\e|-|\Phi_{R,\e}| \ge C_5 R \quad  \text{ on }\{\widetilde V_\e=\widetilde h_\e\}\setminus B_{2R},
\end{equation}
for some  constant $C_5>0$ that does not depend on $\e$ nor $R$. It follows that, by \eqref{proof_lemma_bound_capa_1}, and the classical  estimates between Sobolev capacity and Lebesgue measure, see \cite[Theorem 4.15]{EG_book}, 
\begin{equation}
\limsup_{\e \to 0^+}\int_{\R^d}|\nabla \widetilde V_\e|^2 \, dx 
\ge C_5 R ^2 \limsup_{\e \to 0^+} |\{\widetilde V_\e=\widetilde h_\e\}\setminus B_R|^{\frac{d-2}{d}},
\end{equation}
so that we have proved \eqref{limit_meausure_Ve} if $d \ge 3$. If $d=2$ we can finish the proof in a similar way.

\end{proof}

We are now in position to prove one of the main results of this section.
\begin{proposition}\label{prop_tildeVe_tildeV0}
Suppose that \eqref{hp_h} holds. Then any accumulation point in the weak topology of $D^{1,2}(\R^d, \R^k)$ of $\{\widetilde V_\e\}$ is a minimizer of 
problem \eqref{prob_min_tildeV0}. Letting $\widetilde V_0$ be such a minimizer for a sequence  $\{\widetilde V_{\e_j}\}_{j\in \mb{N}}$, we also have 
\begin{equation}\label{limit_Ve_nabla}
\widetilde V_{\e_j} \to \widetilde V_0, \quad \text{ in  } D^{1,2}(\R^d, \R^k) \text{ and in } L_{loc}^2(\R^d, \R^k) \text{ strongly.}
\end{equation}
In particular,  
\begin{equation}\label{limit_Ve}
\lim_{\e \to 0^+}\frac{1}{\e}\int_D |\nabla V_\e|^2 \, dx =\lim_{\e \to 0^+}\int_{D_\e} |\nabla \widetilde V_\e|^2 \, dx=\int_{\R^d} |\nabla \widetilde V_0|^2 \, dx.
\end{equation}
\end{proposition}

\begin{proof}
By the Poincaré inequality and Lemma \ref{lemma_bound_capa},
$\{\widetilde V_\e\}_{\e \in (0,1)}$ is bounded in $D^{1,2}(\R^d, \R^k)$ and in $H_{loc}^1(\R^d,\R^k)$. In particular, there exists a sequence $V_{\e_j} \rightharpoonup  \widetilde{V}_0$  weakly in $D^{1,2}(\R^d, \R^k)$ as $k \to \infty$ and strongly in $L_{loc}^2(\R^d,\R^k)$  for some $ \widetilde{V}_0 \in D^{1,2}(\R^d, \R^k)$  by a diagonal argument.

We are going to  show that  $ \widetilde{V}_0$ minimizes \eqref{prob_min_tildeV0}. To this end, we need to apply Lemma \ref{limit_psij} to the sequence $h_{\e_j}$, thus  let us check that $h_{\e_j} \to  \nabla h(0) x $ in $H_{loc}^1(\R^d,\R^k)$. For any $R>0$, since $h$ is Lipschitz, the Dominated converge theorem yields
\begin{equation}
\lim_{j \to \infty}\int_{B_R}|h_{\e_j}(x)-\nabla h(0)x|^2 \, dx=0.
\end{equation}
Furthermore, $\nabla h_{\e_j}(x)=(\nabla h)(\e_jx) \to \nabla h(0)$ for any $x \in \mathbb\R^d$, hence, by  the Dominated Converge Theorem we also have 
\begin{equation}
\lim_{j \to \infty}\int_{B_R}|\nabla h_{\e_j}(x)-\nabla h(0)|^2 \, dx=0.
\end{equation}
Let $W\in D^{1,2}(\R^d, \R^k)$ be such that  $|\{W= \nabla h(0)x\}|=1$ and the support of $W$ is compact. For any $j \in \mb{N}$ and $j=\infty$, let $\psi_j$ be as in Lemma \ref{lemma_psij} with $K=\overline{\{W(x)= \nabla h(0)x\}}$,  as $f_j$ the function $h_{\e_j}$ and as $f_\infty$ the function $\nabla h(0) x$.
By Lemma \ref{lemma_psij} and harmonicity
\begin{equation}
\liminf_{j \to \infty} \int_{\R^d} |\nabla \widetilde \psi_j|^2 \, dx=\int_{\R^d} |\nabla \widetilde \psi_\infty|^2 \, dx \le \int_{\R^d} |\nabla W|^2 \, dx.
\end{equation}
It follows that, by lower semicontinuity of norms, and minimality of $\widetilde V_{\e_j}$ 
\begin{equation}
\int_{\R^d} |\nabla \widetilde V_0|^2 \, dx \le \liminf_{j \to \infty}\int_{\R^d} |\nabla \widetilde V_{\e_j}|^2 \, dx \le \liminf_{j \to \infty} \int_{\R^d} |\nabla \widetilde \psi_j|^2 \, dx
\le\int_{\R^d} |\nabla W|^2 \, dx.
\end{equation}
Let us show that 
\begin{equation}\label{proof_prop_tildeVe_tildeV0_1}
|\{\widetilde{V}_0(x)=\nabla h(0) x\}| \ge 1,
\end{equation}
so that  $ \widetilde{V}_0$ minimizes \eqref{prob_min_tildeV0} by Proposition \ref{prop_equiv_ge}. 
Clearly for any $R>0$
\begin{align}
&|\{\widetilde{V}_0(x)=\nabla h(0) x\}\cap B_R|=  |B_R|-  |\{\widetilde{V}_0(x)\neq \nabla h(0) x\}\cap B_R|, \\
&|\{\widetilde{V}_{\e_j}=h_{\e_j}\}\cap B_R|=  |B_R|-  |\{\widetilde{V}_{\e_j}(x)\neq h_{\e_j}\}\cap B_R|,
\end{align}
and
\begin{equation}
\chi_{\{\widetilde{V}_0(y)\neq \nabla h(0) y\}}(x) \le \lim_{j \to \infty} \chi_{\{\widetilde{V}_{\e_j}(x)\neq h_{\e_j}\}}(x).
\end{equation}
Indeed if   $\chi_{\{\widetilde V_\e\neq h_\e\}}(x)=0$ for any $k$ big enough than $\widetilde{V}_{\e_j}(x)= h_{\e_j}(x)$ and so, passing to the limit, as $k \to \infty$, we obtain  $\widetilde{V}_0(x)= \nabla h(0)x$, that is, $\chi_{\{\widetilde{V}_0(y)\neq \nabla h(0) y\}}(x)=0$.
Hence, by Fatou's lemma and Lemma \ref{lemma_bound_capa}, for any $R\ge R_0$
\begin{multline}
|\{\widetilde{V}_0(x)=\nabla h(0) x\}\cap B_R| \ge \limsup_{j \to \infty} |\{\widetilde{V}_{\e_j}=h_{\e_j}\}\cap B_R|\\
=1-\liminf_{j \to \infty} |\{\widetilde{V}_{\e_j}=h_{\e_j}\}\setminus  B_R|\ge 1-\frac{\kappa}{R^2}.    
\end{multline}
Passing to the limit as $R\to +\infty$, we obtain \eqref{proof_prop_tildeVe_tildeV0_1}.
By Proposition \ref{prop_equiv_ge}, $\widetilde{V}_0$ is a minimizer of \eqref{prob_min_tildeV0}.

Next we show \eqref{limit_Ve}. Let $\delta >0$ and let $W\in D^{1,2}(\R^d, \R^k)$ be such that  $|\{W= \nabla h(0)x\}|=1$,  the support of $W$ is compact, and 
\begin{equation}
    \inf\left\{\int_{\R^d} |\nabla W|^2\, dx:W\in D^{1,2}(\R^d, \R^k),  |\{W(x)= \nabla h(0)x\}|=1\right\} \ge \int_{\R^d} |\nabla W|^2\, dx -\delta.
\end{equation}
Such a choice is possible for any $\delta>0$ in view of Proposition \ref{prop_equiv_ge}.
As above 
\begin{equation}
 \int_{\R^d} |\nabla W|^2 \, dx \ge \liminf_{j \to \infty}\int_{\R^d} |\nabla \widetilde V_{\e_j}|^2 \, dx 
\end{equation}
so that 
\begin{multline}
\int_{\R^d} |\nabla \widetilde V_0|^2 \, dx=  \inf\left\{\int_{\R^d} |\nabla W|^2\, dx:W\in D^{1,2}(\R^d, \R^k),  |\{W(x)= \nabla h(0)x\}|=1\right\}\\
\ge\int_{\R^d} |\nabla W|^2\, dx -\delta \ge  \liminf_{j \to \infty}\int_{\R^d} |\nabla \widetilde V_{\e_j}|^2 \, dx -\delta.
\end{multline}
Since this inequality holds for any $\delta >0$, we conclude that, by the lower semicontinuity of norms, 
\begin{equation}
\int_{\R^d} |\nabla \widetilde V_0|^2 \, dx= \liminf_{j \to \infty}\int_{\R^d} |\nabla \widetilde V_{\e_j}|^2 \, dx.
\end{equation}
Hence, we have proved \eqref{limit_Ve} which, combined with the weak convergence in $D^{1,2}(\R^d,\R^k)$, also yields \eqref{limit_Ve_nabla}.
\end{proof}

Let us denote with  $U_{\e,\eta}$  a minimizer $U_{m,\eta}$ of \eqref{prob_min_J_meta}  with  $m=|D|-\e$ and similarly let $\Lambda_{\e,\eta}$ be as in \eqref{eq_shape_derivative}. An interesting consequence of the previous proposition are uniform bounds on $\Lambda_{\e,\eta}$ with respect to both $\eta$ and $\e$.
\begin{proposition}\label{prop_La_aeta_unifor}
If \eqref{hp_h} holds, there exists three constants $C_1,C_2>0$ and $\e_0>0$ such that for any $\e \in (0,\e_0)$ there exists $\eta_\e>0$ such that
\begin{equation}\label{ineq_La_aeta_unifor}
C_1\le \Lambda_{\e,\eta_\e} \le C_2.
\end{equation}
\end{proposition}
\begin{proof}
We argue by contradiction supposing that there exists a sequence $\e_j \to 0^+$ and a sequence $\eta_j\to 0^+$ such that \eqref{ineq_La_aeta_unifor} does not hold, that is,
\begin{equation}\label{proof_prop_La_aeta_unifor_1}
\lim_{j \to \infty}\Lambda_{\e_j,\eta_j}=+\infty \quad \text{ or }\quad \lim_{j \to \infty}\Lambda_{\e_j,\eta_j}=0.
\end{equation}
We may suppose that $\eta_j$ is small enough so that $U_{\e_j,\eta_{j}}$  is a solution to \eqref{prop_minim},  thanks to Proposition \ref{prop_La_meta_unifor}.
By Proposition \ref{prop_shape_for}, letting $U_{\e_j,\eta_j}=(u_{\e_j,\eta_j,1}, \dots, u_{\e_j,\eta_j,k})$, for any $\xi \in C^\infty_c(D,\R^d)$
\begin{equation}
\sum_{i=1}^k \int_D [-2\nabla  u_{\e_j,\eta_j,i}\cdot D\xi   \nabla  u_{\e_j,\eta_j,i}+|\nabla  u_{\e_j,\eta_j,i}|^2 \dive(\xi)] \, dx= -\Lambda_{\e_j,\eta_j} \int_{\Omega_{U_{\e_j,\eta_j}}} \dive(\xi) \, dx.
\end{equation}
A change of variables then yields, for any  $\xi \in C^\infty_c(D_\e,\R^d)$,
\begin{equation}\label{proof_prop}
\sum_{i=1}^k \int_{D_\e} [-2\nabla  \widetilde{u}_{\e_j,\eta_j,i}\cdot D\xi   \nabla  \widetilde u_{\e_j,\eta_j,i}+|\nabla   \widetilde u_{\e_j,\eta_j,i}|^2 \dive(\xi)] \, dx= -\Lambda_{\e,\eta} \int_{\Omega_{ \widetilde U_{\e_j,\eta_j}}} \dive(\xi) \, dx, 
\end{equation}
where $\widetilde{U}_{\e_j,\eta_j}(x):=\e^{-\frac{1}{d}}U_{\e_j,\eta_j}(\e^{\frac{1}{d}} x)$.
By Proposition \ref{prop_tildeVe_tildeV0}, up to passing to a further subsequence,  
\begin{equation}
\widetilde{U}_{\e_j,\eta_j} \to \widetilde{U}_0 \quad \text{ strongly in } H^1_{loc}(\R^d,\R^k),
\end{equation}
where $ \widetilde{U}_0:= \nabla h(0) x -\widetilde{V}_0$ for some minimizer $\widetilde{V}_0$ of \eqref{prob_min_tildeV0}. Indeed, as shown during the proof of Proposition  \ref{prop_tildeVe_tildeV0}, $h_{\e_j} \to \nabla h(0) x  $ strongly in $H^1_{loc}(\R^d,\R^k)$ and for any $R>0$ there exists an $\e_R$ such that $B_R\subset D_\e$ for any $\e \in (0,\e_R)$. 
Hence, we may pass to the limit in \eqref{proof_prop} to conclude that
\begin{equation}
\sum_{i=1}^k \int_{\R^d} [-2\nabla  \widetilde{u}_{0,i}\cdot D\xi   \nabla  \widetilde u_{0,i}+|\nabla   \widetilde u_{0,i}|^2 \dive(\xi)] \, dx= -\left(\lim_{j \to \infty}\Lambda_{\e_j,\eta_j}\right) \int_{\Omega_{ \widetilde U_0}} \dive(\xi) \, dx, 
\end{equation}
for any  $\xi \in C^\infty_c(\R^d,\R^d)$. Indeed,  the pointwise convergence of the characteristic functions $\chi_{\Omega_{ \widetilde U_{\e_j,\eta_j}}}$ to $\chi_{\Omega_{ \widetilde U_0}}$ is guaranteed by the fact that $|\{D_\e\setminus \Omega_{ \widetilde U_{\e_j,\eta_j}}\}|=1$ and, by minimality of $\widetilde{V}_0$, we also have $|\{\R^d\setminus \Omega_{ \widetilde U_0}\}|=1$.
Let $\xi_0 \in C^\infty_c(\R^d,\R^d)$  be  such that 
\begin{equation}
\int_{\Omega_{\widetilde U_0}} \dive(\xi_0) \, dx= 1
\end{equation}
(such a choice is possible, see \cite[Lemma 11.3]{V_book_free_boun}), so that
\begin{equation}
\lim_{j \to \infty}\Lambda_{\e_j,\eta_j}
=-\sum_{i=1}^k \int_{\R^d} [-2\nabla  \widetilde{u}_{0,i}\cdot D\xi_0   \nabla  \widetilde u_{0,i}+|\nabla   \widetilde u_{0,i}|^2 \dive(\xi_0)] \, dx.
\end{equation}
Hence,  $\lim_{j \to \infty}\Lambda_{\e_j,\eta_j}<+\infty$ and so by \eqref{proof_prop_La_aeta_unifor_1} 
\begin{equation}
 \lim_{j \to \infty}\Lambda_{\e_j,\eta_j}=0.
\end{equation}
Then, by Proposition \ref{prop_unique_conti} with $D=B_R$, we conclude that  $|B_R\setminus \Omega_{\widetilde {U}_0}|=0$ for any $R>0$, thus contradicting 
$|\{\R^d\setminus \Omega_{ \widetilde{U}_0}\}|=1$.
\end{proof}

We are now ready to complete the analysis of the asymptotic of minimizers of \eqref{prob_min_measure} as $\e\to 0^+$.
\begin{proof}[\textbf{Proof of Theorem \ref{theo_reg_m_var}}.]
For any $\e>0$, let $U_\e$ be a minimizer of \eqref{prop_minim} with $m:=|D|-\e$. 
We start by showing that  $U_\e \to h$ strongly in  $H^1(D)$ and we notice that the assumption \eqref{hp_h} is not necessary to prove this claim.

Let $V_\e:=U-h$. It is not restrictive to suppose that $0 \in D$. Let $\eta_\e \in C^\infty(\R^d)$ be a cut off function such that $\eta_\e=1$ on 
$B_{r_\e}$ with $r_\e:=(\e/\omega_d)^\frac{1}{d}$, $\eta_\e=0$ in $D\setminus B_{2 r_\e}$ and $|\nabla \eta_\e| \le \frac{2}{r_\e}$.
Then  it holds
\begin{multline}
\int_{D}|\nabla V_\e|^2 \, dx  \le \int_{D}|\nabla w_\e|^2 \, dx \le  
2\int_{D}|\eta_\e|^2|\nabla h|^2 \, dx +2\int_{D}|h|^2|\nabla \eta_\e|^2 \, dx\\
\le C \norm{\nabla h}_{L^\infty(B_{2r},\R^{k,d})} |B_{2 r_\e}|,
\end{multline}
for some dimensional constant $C>0$. Hence, if $d\ge 3$, we have shown that   $U_\e \to h$ strongly in  $H^1(D)$. If $d=2$ we can repeat a similar argument but considering as cut off functions
\begin{equation}
\eta_\e(x):=
\begin{cases}
1, &\text{ if } x \in B_{r^2_\e},\\
\frac{\log(r_\e)-\log(|x|)}{\log(r_\e)-\log(r^2_\e)}, &\text{ if }x \in B_{r_\e}\setminus B_{r^2_\e},\\
0, &\text{ if }x \in \R^d\setminus B_{r_\e}.
\end{cases}
\end{equation}

By Proposition \ref{prop_La_aeta_unifor} and Proposition \ref{prop_Lip} it follows  that $\{U_\e\}$ are bounded in $ C^{0,1}_{loc}(D)$ and so we have proved \eqref{limit_Ue_h}. Furthermore, \eqref{limit_Ue_h_precise} is a consequence of \eqref{eq_Ve_Ue}, \eqref{def_Ve_he}, Proposition \ref{prop_tildeVe_tildeV0} and the minimality of $U_\e$.

Suppose that $x_\e \in |D\setminus \Omega_{U_\e}|$ and that $\mathop{\rm dist}(x_\e, \partial D) \ge \delta$ for some $\delta>0$ that does not depend on $\e$. Then, up to subsequences, $x_\e \to \bar x$ for some $\bar x \in D$. Let us show that $h(\bar x)=0$ so that $\bar x=x_0$. Let $\eta>0$ and $y \in D$ 
such that $|y-\bar x| \le \eta$ 
\begin{multline}
|h(\bar x)|\le |h(\bar x)-h(y)|+|h(y)-U_\e(y)| +|U_\e(y)-U_\e(x_\e)| \\
\le \norm{\nabla h}_{L^\infty(B_\delta(\bar x),\R^{k,d})} |\bar x-y|+ C|y-x_\e| +o(1)
\le \norm{\nabla h}_{L^\infty(B_\delta(\bar x),\R^{k,d})}  \eta + C\eta +o(1), \text{ as }\e \to 0^+
\end{multline}
for some constant $C>0$, since $\{U_\e\}$ is bounded in $ C^{0,1}_{loc}(D)$. Hence, we conclude that $h(\bar x)=0$ since $\eta>0$ is arbitrary and so  the limit in local Hausdorff sense in $D$ of  $D \setminus \Omega_{U_\e}$ is $\{x_0\}$.
\end{proof}

Without assuming \eqref{hp_h}, it is not true in general that $\{U_\e\}$ are locally equi-Lipschitz in $D$ as the next example shows.
\begin{example}\label{exam_not_lip}
Let us consider minimizers $\{u_\e\}$ of 
\begin{equation}
\inf\left\{\int_{B_1} |\nabla w|^2 \, dx : w \in H^1(B_1), w-1 \in H_0^1(B_1), |\Omega_w|=|B_1|-\e\right\}.
\end{equation}
With a symmetrical rearrangement, we can see that $u_\e$ is radial and, letting $r_\e:=(\e/\omega_d)^{\frac{1}{d}}$,
\begin{equation}
u_\e(x):=
\begin{cases}
\frac{r_\e^{2-d}-|x|^{2-d}}{r_\e^{d-2}-1}, &\text{ if } d\ge 3,\\
\frac{\log(r_\e)-\log(|x|)}{\log(r_\e)} &\text{ if } d=2.
\end{cases}
\end{equation}
Hence
\begin{equation}
\pd{u_\e}{r}(r_\e)=
\begin{cases}
\frac{-(2-d)r_\e^{1-d}}{r_\e^{2-d}-1}, &\text{ if } d\ge 3,\\
\frac{1}{r_\e\log(r_\e)} &\text{ if } d=2.
\end{cases}
\end{equation}
Since, 
\begin{equation}
-(2-d)\frac{r_\e^{1-d}}{r_\e^{d-2}-1}=-\frac{(2-d)}{r_\e}+o(1), \quad \text{ as } \e \to 0^+,
\end{equation}
we conclude that $\{u_\e\}$ are not locally equi-Lipschitz in $B_1$, and actually are not even equi-continuous.
\end{example}

\section{The linear case}\label{sec_linear}
In this section,  we analyze in more detail the case in which the boundary datum $g$ in \eqref{prob_min_measure} is given by  $g(x)=Ax$, with $A$ a $k \times d$ constant matrix and $D=B_1$. In this case, the harmonic extension of $g$ to $B_1$ is simply $Ax$ and \eqref{hp_h} translates into
\begin{equation}\label{hp_A}
 \mathop{\rm{Ker}}(A)=\{ 0\} \quad \text{ or equivalently } \quad {\rm rk} (A)=d
\end{equation}
and if this condition is satisfied  all the results of Section \ref{sec_uni} holds. Under the  more general assumption
\begin{equation}\label{hp_A_sharp}
 \quad {\rm rk} (A)>1,
\end{equation}
we  study the optimal constant \eqref{def_La}. This is not restrictive since, if  ${\rm rk} (A)=1$, it is already know that $\Lambda^*(A)=\norm{A}^2$. Furthermore, we  study in more details the properties of minimizers of \eqref{prob_min_boundedness}.

\subsection{Computation of $\Lambda^*(A)$}
With the same notations of Section \ref{sec_uni}, our first result is the following.
\begin{proposition}\label{prop_lim_inf_We}
We have that
\begin{multline}\label{lim_inf_We}
\Lambda^*(A)=\inf\left\{\int_{\R^d} |\nabla W|^2\, dx:W\in D^{1,2}(\R^d, \R^k),  |\{W(x)= Ax\}|=1\right\}\\
=\lim_{\e \to 0^+}\inf\left\{\int_{B_{1/\e}} |\nabla W|^2\, dx:W\in  H_0^1(B_{1/\e}, \R^k),  |\{W(x)=Ax\}|=1\right\}.
\end{multline}
\end{proposition}
\begin{proof}
Let $\delta >0$ and let $W\in D^{1,2}(\R^d, \R^k)$ be such that $|\{W(x)= \nabla Ax\}|=1$ with compact support  and 
\begin{equation}
    \inf\left\{\int_{\R^d} |\nabla W|^2\, dx:W\in D^{1,2}(\R^d, \R^k),  |\{W(x)= \nabla Ax\}|=1\right\} \ge \int_{\R^d} |\nabla W|^2\, dx -\delta.
\end{equation}
It follows that, as soon as the support of $W$ is in $B_{1/\e}$, 
\begin{multline}
 \inf\left\{\int_{\R^d} |\nabla W|^2\, dx:W\in D^{1,2}(\R^d, \R^k),  |\{W(x)= \nabla Ax\}|=1\right\}\\
\ge\int_{\R^d} |\nabla W|^2\, dx -\delta \\
\ge  \inf\left\{\int_{B_{1/\e}} |\nabla W|^2\, dx:W\in H_1^0(_{B_{1/\e}}, \R^k),  |\{W(x)= \nabla Ax\}|=1\right\}-\delta.
\end{multline}
Passing to the limit as $\e\to 0^+$, since $\delta$ is arbitrary, we have one inequality in the second equality of   \eqref{lim_inf_We}.  The reverse inequality is trivial.

Next we show that \eqref{eq_La} holds.
The function $Ax$ minimizes \eqref{prob_min_vec} with $R=1$ for some constant $\Lambda>0$ if and only if for any $W \in H^1_0(B_1)$
\begin{equation}
\int_{B_1}|\nabla Ax|^2 dx +\Lambda |B_1|\le \int_{B_1}|\nabla (W-Ax)|^2 dx +\Lambda |\Omega_{W-A}|
\end{equation}
or equivalently, since $Ax$ is harmonic, 
\begin{equation}
\int_{B_1}|\nabla W|^2 dx \ge \Lambda (|B_1|-|\Omega_{W-A}|)=\Lambda |\{W(x)=Ax\}|
\end{equation}
which proves  \eqref{eq_La}.

If $\Lambda^*(A)$, e.g. the infimum in \eqref{eq_La}, coincides with the limit as $\e \to 0^+$, then the first equality in \eqref{lim_inf_We} is a consequence of the second one and the same change of variables  of \eqref{def_Ve_he}.

Otherwise, there exist $\e_0>0$, a sequence $\e_j \to \e_0$  as $j \to \infty$, and a sequence  $V_{\e_j} \in H_0^1(B_1, \R^k)$ such that $V_{\e_j}$ minimizes 
\begin{equation}
\inf\left\{\frac{1}{\e_j}\int_{B_1} |\nabla W|^2 \, dx:  W \in H^1_0(B_1,\R^k), |\{W(x)=Ax\}|=\e_j\right\},
\end{equation}
and 
\begin{equation}
\Lambda^*(A)=\lim_{j \to \infty}\frac{1}{\e_j}\int_{B_1} |\nabla  V_{\e_j}|^2 \, dx.
\end{equation}
Up to passing to a subsequence, there exists $V_0 \in H_0^1(B_1, \R^k)$ such that  $V_{\e_j} \rightharpoonup V_0$ weakly in $H_0^1(B_1, \R^k)$ as $j \to \infty$.
Then by the lower semicontinuity of norms
\begin{equation}
\int_{B_1} |\nabla  V_0|^2 \, dx\le \liminf_{j \to \infty}\int_{B_1} |\nabla  V_{\e_j}|^2 \, dx
\end{equation}
and it is easy to check that 
\begin{equation}
|\{V_0(x)=Ax\}| \ge \e_0.
\end{equation}
Then, defining $\widetilde{V}_0(x):=\e_0^{-1/d}V_0(\e_0^{1/d}x)$, by Proposition \ref{prop_equiv_ge}
\begin{multline}
\inf\left\{\int_{\R^d} |\nabla W|^2\, dx:W\in D^{1,2}(\R^d, \R^k),  |\{W(x)= Ax\}|=1\right\} \\
\le \int_{B_{1/\e_0}} |\nabla  \widetilde{V}_0|^2 \, dx=\frac{1}{\e_0}\int_{B_1} |\nabla  V_0|^2 \, dx \le \liminf_{j \to \infty}\frac{1}{\e_j}\int_{B_1} |\nabla  V_{\e_j}|^2 \, dx =\Lambda^*(A).
\end{multline}
On the other hand, 
\begin{multline}
\Lambda^*(A)= \inf_{\e \in (0,|B_1|)} \inf \left\{\frac{1}{\e}\int_{B_1} |\nabla W|^2\, dx:W\in  H_0^1(B_1, \R^k),  |\{W(x)=Ax\}|=\e\right\} \\
= \inf_{\e \in (0,|B_1|)}\inf\left\{\int_{B_{1/\e}} |\nabla W|^2\, dx:W\in  H_0^1(B_{1/\e}, \R^k),  |\{W(x)=Ax\}|=1 \right\} \\
\ge \inf\left\{\int_{\R^d} |\nabla W|^2\, dx:W\in D^{1,2}(\R^d, \R^k),  |\{W(x)=Ax\}|=1 \right\}.
\end{multline}
In conclusion, also in this case
\begin{equation}
\Lambda^*(A)= \inf\left\{\int_{\R^d} |\nabla W|^2\, dx:W\in D^{1,2}(\R^d, \R^k),  |\{W(x)=Ax\}|=1 \right\}
\end{equation}
thus we have  proved   \eqref{lim_inf_We}.
\end{proof}
Since we are not assuming  \eqref{hp_h} (or equivalently \eqref{hp_A}), it is not clear at all that \eqref{prob_min_boundedness} has  a minimizer. Indeed, the proof of Lemma \ref{lemma_bound_capa} does not work without    \eqref{hp_h} and Lemma \ref{lemma_bound_capa} is  a crucial ingredient in the proof of Proposition \ref{prop_tildeVe_tildeV0}. Hence, whenever ${\rm rk}(A)<d$ to compute $\Lambda^*(A)$ we  follow a different  strategy based on a reduction argument (to the case of maximal rank) by a smart choice of comparators.

\begin{proof}[\textbf{Proof of Theorem \ref{theo_La_computed}.}] 
Suppose that ${\rm rk}(A)=n$ with $n \in \{2,\dots,d\}$. Let $\{\tau_i\}_{i=1,\dots d-n}$ be an orthonormal basis  of the kernel of $A$ and  let us  complete it to an orthonormal  basis $\{\tau_i\}_{i=1,\dots d}$  of $\R^d$.  Defining   the orthogonal matrix  $Q \in \R^{d,d}$ as the matrix with the vectors $\{\tau_i\}_{i=1,\dots d}$  as rows,  
$A=\begin{bmatrix} A_1, 0\end{bmatrix}Q$, where $A_1 \in \R^{k,n}$ has rank $n$.

Furthermore, for any $W\in D^{1,2}(\R^d, \R^d)$, the change of variables $x=Q^Tx'$ yields,
letting $W_Q(x'):=W(Q^Tx')$,
\begin{equation}
\int_{\R^d} \nabla W\cdot \nabla W\, dx=\int_{\R^d} Q^T\nabla W \cdot Q^T\nabla W\, dx=\int_{\R^d}\nabla W_Q\cdot \nabla W_Q\, dx',
\end{equation}
while 
\begin{equation}
\{W(x)=Ax\}=\{W_Q(x')=\begin{bmatrix} A_1, 0\end{bmatrix}x'\}. 
\end{equation}
Hence,
\begin{multline}
\inf\left\{\int_{\R^d} |\nabla W|^2\, dx:W\in D^{1,2}(\R^d, \R^k),  |\{W(x)= Ax\}|=1\right\}\\
=\inf\left\{\int_{\R^d} |\nabla W|^2\, dx:W\in D^{1,2}(\R^d, \R^k),  |\{W(x)= \begin{bmatrix} A_1, 0\end{bmatrix}x\}|=1\right\},
\end{multline}
thus  by Proposition \ref{prop_lim_inf_We}, we only have to show that 
\begin{multline}\label{proof_theo_La_computed_0_5}
\inf\left\{\int_{\R^d} |\nabla W|^2\, dx:W\in D^{1,2}(\R^d, \R^k),  |\{W(y)= \begin{bmatrix} A_1, 0\end{bmatrix}x\}|=1\right\}=\\
=\inf\left\{\int_{\R^n} |\nabla W|^2\, dy:W\in D^{1,2}(\R^n, \R^k),  |\{W(y)= A_1y\}|=1\right\}.
\end{multline}
Let  $\R^d=\R^n \times \R^{d-n}$ with variables $x=(y,z) \in \R^n \times \R^{d-n}$. Let us define the $n$-dimensional hyperplane
$\pi_z:=\R^n \times \{z\}$.
We  also denote with $B_r'$ and $B_r''$ the balls of radius $r$ and center $0$ in $\R^n$ and $\R^{d-n}$ respectively.
We divide  the proof of \eqref{proof_theo_La_computed_0_5} in two steps.

\textbf{Step 1. Lower bound.}
Let $W\in D^{1,2}(\R^d, \R^k)$ with   $|\{W(x)=\begin{bmatrix} A_1, 0\end{bmatrix}x\}|=1$. To simplify the notations, let  $K:=\{W(x)= \begin{bmatrix} A_1, 0\end{bmatrix}x\}$ and $K_z:=K \cap \pi_z$.
For any $z$ such that $\mathcal{H}^n(K_z)>0$, let us define
\begin{equation}
W_z(y):=\frac{1}{(\mathcal{H}^n(K_z))^\frac{1}{n}} W(\mathcal{H}^n(K_z))^\frac{1}{n} y,z).
\end{equation}
If $\int_{\R^n} |\nabla_y W|^2\, dy  <+\infty$, that is, for  a.e. $z \in \R^{d-n}$, a change of variables yields
\begin{equation}
\int_{\R^n} |\nabla_y W|^2\, dy  =\mathcal{H}^n(K_z))\int_{\R^n} |\nabla_y W_z|^2\, dy.
\end{equation}
Furthermore, since $|\{W_z (y)= A_1 y \}|=1$,
\begin{equation}
\int_{\R^n} |\nabla_y W|^2\, dy  \ge \mathcal{H}^n(K_z))\int_{\R^n} |\nabla_y V|^2\, dy,
\end{equation}
where  $V$ is a minimizer of the right hand side of \eqref{eq_theo_La_computed}.
Hence, 
\begin{multline}
\int_{\R^d} |\nabla W|^2\, dx  \ge \int_{\R^d} |\nabla_y W|^2\, dx=\int_{\R^{n-d}}\int_{\R^n} |\nabla_y W|^2\, dy \, dz\\
\ge\int_{\R^{n-d}}\int_{\R^n} \mathcal{H}^n(K_z))|\nabla_y V|^2\, dy \, dz= \int_{\R^n}|\nabla_y V|^2\, dy\\
=\min\left\{\int_{\R^n} |\nabla W|^2\, dx:W\in D^{1,2}(\R^n, \R^k),  |\{W(y)= A_1y\}|=1\right\},
\end{multline}
since $\int_{\R^{n-d}} \mathcal{H}^n(K_z)\ dz =|K|=1$. Hence, since $W$ is an arbitrary comparator for \eqref{prob_min_tildeV0}, 
\begin{multline}
\inf\left\{\int_{\R^d} |\nabla W|^2\, dx:W\in D^{1,2}(\R^d, \R^k),  |\{W(x)= Ax\}|=1\right\}\\
\ge  \min\left\{\int_{\R^n} |\nabla W|^2\, dy:W\in D^{1,2}(\R^n, \R^k),  |\{W(y)= A_1y\}|=1\right\}.
\end{multline}
\textbf{Step 2. Upper bound.} Let $\varphi \in C^\infty_c(\R^{d-n})$ be a cut off function such that $\varphi=1$ in $B_1''$. For any $\delta>0$, let us define 
\begin{equation}
\varphi_\delta(z)=\varphi(\omega_{d-n}^{\frac{1}{d-n}}\delta^{\frac{1}{d-n}} z).
\end{equation}
Let us consider a sequence $V_j$  of minimizers of \eqref{prob_min_tildeVe} with $\e=\e_j$ and $\e_j \to 0^+$ in dimension $n$ and with $h_j=A_1$
for any $j \in \mathbb{N}$. Furthermore,  by Proposition \ref{prop_tildeVe_tildeV0}, we may choose $V_j$ so that $V_j \to V$ strongly in $D^{1,2}(\R^n,\R^k)$ as $j\to \infty$ where $V$ is a minimizer of  the right hand side of \eqref{eq_theo_La_computed}.
Let $V_{j,\delta}(y):=\delta^{\frac{1}{n}} V_j(y/\delta^{\frac{1}{n}})$ for any $\delta>0$ and
\begin{equation}
W_{\delta,j,\varphi}(y,z):= V_{j,\delta}(y)\varphi_\delta(z).
\end{equation}
If we define  $K_j:=\{V_j(y)=A_1y\}$ and $r_\delta:=\omega_{d-n}^{-\frac{1}{d-n}}\delta^{-\frac{1}{d-n}}$, 
\begin{equation}
 V_{j,\delta}(y)\varphi_\delta(z)=\begin{bmatrix} A_1, 0\end{bmatrix}(y,z) \quad \text{ if  } \quad y \in \delta^\frac{1}{n} K_j \text{ and } z \in B''_{r_\delta}
\end{equation}
so that 
\begin{equation}
|\{W_{\delta,j,\varphi}(x)=\begin{bmatrix} A_1, 0\end{bmatrix} x\}| \ge  \delta  | B''_{r_\delta}|= 1.
\end{equation}
Then, by Proposition \ref{prop_equiv_ge},
\begin{equation}
\inf\left\{\int_{\R^d} |\nabla W|^2\, dx:W\in D^{1,2}(\R^d, \R^k),  |\{W(x)= \begin{bmatrix} A_1, 0\end{bmatrix} x |=1\right\}
\le   \int_{\R^d} |\nabla W_{\delta,j,\varphi}|^2\, dx.
\end{equation}
Furthermore, the changes of variables $y'=\delta^{-\frac{1}{n}} y$ and $z'=\omega_{d-n}^{\frac{1}{d-n}}\delta^{\frac{1}{(d-n)}} z$ yield 
\begin{multline}
\int_{\R^d} |\nabla W_{\delta,j,\varphi}|^2\, dx= \int_{\R^{d-n}} |\varphi_\delta|^2 \, dz \int_{\R^n}  |\nabla V_{j,\delta}|^2 \, dy +
\int_{\R^{d-n}} |\nabla \varphi_\delta|^2 \, dz \int_{\R^n}  |V_{j,\delta}|^2 \, dy\\
=\frac{1}{\omega_{d-n}}\int_{\R^{d-n}} | \varphi|^2 \, dz  \int_{\R^n}  |\nabla V_j|^2 \, dy+\omega_{d-n}^{-1-\frac{2}{d-n}}\delta^{\frac{2}{d-n}+\frac{2}{n}} \int_{\R^{d-n}} |\nabla \varphi|^2 \, dz \int_{\R^n}  |V_j|^2 \, dy.
\end{multline}
Passing to the limit as $\delta \to 0^+$ and then as $j \to \infty$, we obtain
\begin{multline}
\inf\left\{\int_{\R^d} |\nabla W|^2\, dx:W\in D^{1,2}(\R^d, \R^k),  |\{W(x)= \begin{bmatrix} A_1, 0\end{bmatrix} x |=1\right\}\\
\le \frac{1}{\omega_{d-n}} \int_{\R^{d-n}} |\varphi|^2 \, dz \int_{\R^n}  |\nabla V|^2 \, dy
\end{multline}
for any  cut off function $\varphi \in C^\infty_c(\R^{d-n})$ such that $\varphi=1$ in $B_1''$. Hence, we may take a sequence $\varphi_j$ such that $\varphi_j=1$ in $B''_{1+1/j}$ and $\varphi_j=0$  in $\R^d\setminus B''_{1+2/j}$ and pass to the limit as $j \to \infty$ to show that 
\begin{equation}
\inf\left\{\int_{\R^d} |\nabla W|^2\, dx:W\in D^{1,2}(\R^d, \R^k),  |\{W(x)= \begin{bmatrix} A_1, 0\end{bmatrix} x |=1\right\}\le  \int_{\R^n}  |\nabla V|^2 \, dy.
\end{equation}
In conclusion, we have proved that
\begin{multline}
\inf\left\{\int_{\R^d} |\nabla W|^2\, dx:W\in D^{1,2}(\R^d, \R^k),  |\{W(x)=\begin{bmatrix} A_1, 0\end{bmatrix} x |=1\right\}\\
\le  \min\left\{\int_{\R^n} |\nabla W|^2\, dy:W\in D^{1,2}(\R^n, \R^k),  |\{W(y)= A_1y\}|=1\right\},
\end{multline}
which, combined with the reverse inequality proved in Step 1, yields \eqref{eq_theo_La_computed}.

As  mentioned in Section \ref{sec_intro}, \eqref{eq_La_rk1} has  already been  proven  in \cite[Section 2D]{MTV_reg_vect}. On the other hand, if $\mathop{\rm rk}(A)>1$, then 
\begin{multline}
\Lambda^*(A)=\int_{\R^n} |\nabla V|^2 \, dy >\int_{\{V(y)=Ay\}} |\nabla V|^2 \, dy =\norm{A_1}^2=\norm{[A_1,0]}^2 \\
=\Tr([A_1,0]^T[A_1,0])= \Tr(Q^T[A_1,0]^T[A_1,0]Q)=\Tr(A^TA)=\norm{A}^2,
\end{multline}
where $V$ is a minimizer of \eqref{eq_theo_La_computed}. Hence, we have also proved \eqref{eq_La_rk_bigger_1}.
\end{proof}

\subsection{$\Lambda^*(A)$ for blows-up of minimizers of the vectorial Bernoulli problem}
Let $U$ be a minimizers of \eqref{prob_min_vec} and let $x_0 \in \mathrm{Sing}_2(\partial \Omega_{U}) :=  \Omega^{(1)}_{U} \cap \partial  \Omega_{U} \cap D$.
The constant $\Lambda^*(A)$ defined in \eqref{def_La} turns out to depend only on $x_0$ and not on $A \in \mc{BU}_U(x_0)$. To prove this result, we need a preliminary lemma.

\begin{lemma}\label{lemma_dependence_A}
Let $A$ be a $k \times d$ matrix with $\mathop{\rm{rk}}(A)=d$. Then 
\begin{multline}
\inf\left\{\int_{\R^d} |\nabla W|^2\, dy:W\in D^{1,2}(\R^d, \R^k),  |\{W(x)= Ax\}|=1\right\}\\
=\inf\left\{\int_{\R^d} B\nabla W \cdot \nabla  W\, dy:W\in D^{1,2}(\R^d, \R^d),  |\{W(x)= x\}|=1\right\},
\end{multline}
where $B \in \R^{d,d}$ is the diagonal matrix with the eigenvalues of $A^TA$, counted with multiplicity, as entries on  its diagonal.
\end{lemma}
\begin{proof}
The condition $W(x)= Ax$ is equivalent to $A^{-1}W(x)=x$ and so 
\begin{multline}
\inf\left\{\int_{\R^d} |\nabla W|^2\, dy:W\in D^{1,2}(\R^d, \R^k),  |\{W(x)= Ax\}|=1\right\}\\
=\inf\left\{\int_{\R^d} A^TA\nabla W\cdot \nabla W\, dy:W\in D^{1,2}(\R^d, \R^d),  |\{W(x)= x\}|=1\right\}.
\end{multline}
Furthermore, since $A^TA$ is symmetric, there exists an orthonormal matrix $Q \in \R^{d,d}$ such that $A^TA=Q^TBQ$. For any $W\in D^{1,2}(\R^d, \R^d)$, the change of variables $x=Qx'$ yields,
letting $W_Q(x'):=W(Qx')$,
\begin{equation}
\int_{\R^d} A^TA\nabla W\cdot \nabla W\, dx=\int_{\R^d}B Q\nabla W \cdot Q\nabla W\, dx=\int_{\R^d}B\nabla W_Q\cdot \nabla W_Q\, dx',
\end{equation}
thus completing the proof.
\end{proof}

\begin{proof}[\textbf{Proof of Theorem \ref{theo_dependence_LA}.}]
We first observe that the rank of a blow-up limit  $A$ of $U$ at $x$ depends only on the point $x$, see \cite[Lemma 4.2]{MTV_reg_vect}.  
Letting for any $\sigma \in \R^k$
\begin{equation}
\Phi(U,r,x_0,\sigma):=\frac{1}{r^4}\int_{B_r(x_0)\cap \{\sigma \cdot U>0\}} \frac{|\nabla( U\cdot \sigma)|^2}{|x-x_0|^{d-2}} \, dx
\int_{B_r(x_0)\cap \{\sigma \cdot U<0\}} \frac{|\nabla( U\cdot \sigma)|^2}{|x-x_0|^{d-2}} \, dx
\end{equation}
by the Alt-Cafferelli-Friedman Monotonicity formula,  $\Phi(U,r,x_0,\sigma)$ is not-decreasing in $r$. As proved in   \cite[Lemma 3.1]{PESV_rec} this implies that $|A^T\sigma|=|(A')^T\sigma|$ for any $\sigma \in \R^k$ and any blow-up limit $A,A' \in \mc{BU}_U(x_0)$. 
It follows that for any $\sigma \in \R^k$
\begin{equation}
\sigma\cdot AA^T \sigma=A^T\sigma \cdot A^T \sigma=|A^T\sigma|^2=|(A')^T\sigma|^2= \sigma\cdot A'(A')^T \sigma
\end{equation}
thus $AA^T=A'(A')^T$. It follow that there exists an orthogonal matrix $Q\in \R^d$ such that $A=A'Q$ and so  $A^TA=Q^T(A')^TA'Q$. In particular $A^TA$ and $(A')^TA'$ have the same eigenvalues thus also the same trace.
Then, by  Theorem \ref{theo_La_computed} and Lemma \ref{lemma_dependence_A} (applied in dimension $n=\mathop{\rm{rk}}(A)$ to $A_1$), we conclude the constant $\Lambda^*(A)$ defined in \eqref{def_La} depend only on $x_0$ if $\mathop{\rm{rk}}(A)>1$.
If instead $\mathop{\rm{rk}}(A)=1$, then we notice that 
\begin{equation}\label{eq_norm_A_tr}
\norm{A}^2=\sum_{i=1}^k\sum_{j=1}^d a_{i,j}^2 =\sum_{i=1}^d(A^TA)_{i,i}=\Tr(A^TA),
\end{equation}
thus also in this case $\Lambda^*(A)$ depends only on $x_0$.

Since   $\Phi(U,r,x_0,\sigma)$ is not-decreasing in $r$ and continuous in $x_0$, it follows that 
$\Phi(U,x_0,\sigma):=\lim_{r\to 0^+}\Phi(U,r,x_0,\sigma)$ is upper semicontinuos in $x_0$. Since $\Phi(U,x_0,\sigma)=c_d|A^T\sigma|$, where $c_d$ is a dimensional constant,  for any   $A\in \mc{BU}_U(x_0)$ (see the proof of  \cite[Lemma 3.1]{PESV_rec}),  if $x_n \to x_0$ with $x_n \in \mathrm{Sing}_2(\partial \Omega_{U})$ and $A_n \in \mc{BU}_U(x_n)$ then for any $\sigma \in \R^k$
\begin{equation}
\lim_{n\to \infty}A_nA_n^T \sigma \cdot \sigma \le AA^T \sigma \cdot \sigma.
\end{equation}
By the min-max characterization of the eigenvalues $\{\la_i(A_nA_n^T)\}$ of the symmetric matrix  $A_nA_n^T \in \R^{k,k}$ it follows  that 
\begin{equation}
\limsup_{n \to \infty}\la_i{(A_nA_n^T)}\le \la_i(AA^T),
\end{equation}
where $\{\la_i(AA^T)\}$ are the eigenvalues of $AA^T \in \R^{k,k}$. Since the positive eigenvalues of $AA^T$ and $A^TA$ coincides, by Theorem \ref{theo_La_computed},  Lemma   \ref{lemma_dependence_A} and \eqref{eq_norm_A_tr} we conclude that  $x_0 \mapsto \Lambda^*$  is upper semicontinuos. 
\end{proof}

\subsection{Boundedness of the contact set}\label{subsec_bound_min}
In this subsection, we prove Theorem \ref{theo_V} using the minimality of $V$, harmonic replacement on the complementary of balls  and an iteration procedure in the spirit of De Giorgi regularity theory. 

\begin{proof}[\textbf{Proof of Theorem \ref{theo_V}}]
Let $V$ be a  minimizer of \eqref{prob_min_boundedness} and $R>0$. 
Similarly to  Lemma \ref{lemma_bound_capa}, let us  $\Phi_R$ be  the harmonic extension $\Phi_R$ of $V$ from $\partial B_R$ in $\R^d\setminus B_R$, that is, 
\begin{equation}\label{def_Phi_R}
\Phi_R(x)=\frac{|x|^2-R^2}{d \omega_d R} \int_{\partial B_R} \frac{V(\xi)}{|\xi-x|^d} \, d\mc{H}^{d-1}_\xi \quad 
\text{ for any } x \in \R^d\setminus \overline{B_R}
\end{equation}
and let 
\begin{equation}
W_R:=
\begin{cases}
V, &\text{ in }B_R,\\
\Phi_R, &\text{ in }\R^d \setminus B_R.
\end{cases}
\end{equation}
Since $|\{V(x)= Ax\}|=1$, there exists $R_0>0$ such that 
$\delta_R:=|\{V(x)=Ax\}\setminus B_R| \in [0,1/2)$ for any $R \ge R_0$. Then we may define 
\begin{equation}
\widetilde W_R(x):=\frac{1}{(1-\delta_R)^{1/d}}W_R((1-\delta_R)^{1/d} x).
\end{equation}
By definition, $|\{\widetilde W_R(x)=Ax\}|\ge 1$ thus  by minimality  of $V$ and a change of variables
\begin{multline}
\int_{\R^d}|\nabla V|^2 \, dx \le  \int_{\R^d}|\nabla \widetilde W_R|^2 \, dx 
=\frac{1}{1-\delta_R}\int_{\R^d}|\nabla W_R|^2\,dx\\
=\frac{1}{1-\delta_R}\left(\int_{\R^d}|\nabla V|^2 \, dx-\int_{\R^d}(|\nabla V|^2-|\nabla W_R|^2) \, dx\right).
\end{multline}
It follows that for any $R\ge R_0$, by harmonicity of $\Phi_R$,
\begin{multline}\label{proof_theo_V_contact_1}
\int_{\R^d\setminus B_R}(|\nabla (V-\Phi_R)|^2) \, dx=\int_{\R^d\setminus B_R}(|\nabla V|^2-|\nabla \Phi_R|^2) \, dx\\
\le \frac{\delta_R}{1-\delta_R}\int_{\R^d}|\nabla V|^2 \, dx\le2 \delta_R\int_{\R^d}|\nabla V|^2 \, dx.
\end{multline}
Just as in Lemma \ref{lemma_bound_capa}, we can show that for any $\delta>0$, eventually choosing a larger $R_0$ depending on $\delta$, for any $R\ge R_0$ and any $\rho \ge R$
\begin{equation}
|\Phi_R(x)| \le \delta( R+\rho) \quad \text{ for any } x \in \R^d\setminus B_{R+\rho}
\end{equation}
and 
\begin{equation}
|V(x)|=|Ax| \ge (R+\rho) \min_{x \in \partial B_1}|Ax| \quad \text{ on }\{V(x)=Ax\}\setminus B_{R+\rho}.
\end{equation}
It follows that, up to choosing $\delta <  \min_{x \in \partial B_1}|Ax|$,
there exists a constant $C>0$, that does not depend on $R$ such that for any $R \ge R_0$, 
\begin{equation}
|V(x)-\Phi_R(x)| \ge C(R+\rho) \ge C\rho \quad \text{ for any } x \in \{V=A\}\setminus B_{R+\rho}.
\end{equation}
Suppose now that $d \ge 3$. By \cite[Theorem 4.15]{EG_book} and \eqref{proof_theo_V_contact_1} for any $R\ge R_0$
\begin{equation}\label{proof_theo_V_contact_2}
|\{V=A\}\setminus B_{R+\rho}| \le \frac{C_1}{\rho^\frac{2d}{d-2}} |\{V=A\}\setminus B_R|^{\frac{d}{d-2}}
\end{equation}
for some constant $C_1>0$ that does not depend on $R$ nor $\rho$.
Let us consider the sequence 
\begin{equation}
R_n:= R_0+ \left(1-\frac{1}{2^n}\right) \rho \quad \text{ for } n \in \mathbb{N}\setminus\{0\}.
\end{equation}
We are going to iterate \eqref{proof_theo_V_contact_2}  in the spirit of De Giorgi regularity theory along the sequence $\{R_n\}$.
For any $n \in \mathbb{N}$
\begin{equation}
|\{V=A\}\setminus B_{R_{n+1}}| \le \frac{C_1}{(4^{-n}\rho^2)^\frac{d}{d-2}} |\{V=A\}\setminus B_{R_n}|^{\frac{d}{d-2}}
\end{equation}
thus iterating we obtain the estimate 
\begin{equation}
|\{V=A\}\setminus B_{R_n}| \le\left( \frac{C_1^n}{\rho^\frac{2dn}{d-2}} 
\prod_{j=0}^{n-1}4^{(j+1)\left(\frac{d}{d-2}\right)^{n-j}}\right)|\{V=A\}\setminus B_{R_0}|^{\frac{dn}{d-2}}.
\end{equation}
Furthermore 
\begin{equation}
\prod_{j=0}^{n-1}4^{(j+1)\left(\frac{d}{d-2}\right)^{n-j}}=4^{\sum_{j=0}^{n-1}(j+1)\left(\frac{d}{d-2}\right)^{n-j}},
\end{equation}
and, letting $r:=\frac{d}{d-2}$,
\begin{multline}
\sum_{j=0}^{n-1}(j+1)r^{n-j}=\sum_{k=1}^n(n-k+1)r^k=(n+1)\frac{r(1-r^n)}{1-r}-\frac{r\big(1-(n+1)r^n+n r^{\,n+1}\big)}{(1-r)^2}\\
=\frac{r\big(r^{\,n+1}-(n+1)r+n\big)}{(1-r)^2}.
\end{multline}
Hence, taking $\rho$ large enough and letting $n \to \infty$ we conclude that 
\begin{equation}
|\{V=A\}\setminus B_{R_0+\rho}|=\lim_{n \to \infty}|\{V=A\}\setminus B_{R_n}|=0,
\end{equation}
thus $\{V(x)=Ax\}$ is  bounded if $d \ge 3$. If $d=2$ the proof is similar.

Hence, in any dimension $d\ge 2$ we have proved that there exists $\bar{R}>0$ such that 
\begin{equation}
|\{V=A\}\setminus B_{\bar R}|=0.
\end{equation}
It follows that $V-A$ is a solution of \eqref{prob_min_measure_A} in $B_R$ for any $R> \bar R$ and so $V \in C^{0,1}_{loc}(\R^d,\R^k)$ by Theorem \ref{theo_reg_m_fixed}.
Furthermore, $V$ is harmonic on $\R^d\setminus B_{\bar R}$ so that $V=\Phi_{\bar R}$ on $\R^d\setminus B_{\bar R}$.
In view of  \eqref{def_Phi_R} on  $\R^d\setminus B_{\bar R}$ we have 
\begin{equation}
|\Phi_{\bar R}(x)| \le \frac{C}{|x|^{d-2}}\quad \text{ and } \quad |\nabla \Phi_{\bar R}| \le \frac{C}{|x|^{d-1}},
\end{equation}
for some positive constant $C>0$ depending only on $d$, $\bar R$ and $\norm{V}_{L^\infty(\partial B_{\bar R})}$. Hence, we have completed the proof.
\end{proof}

\section{A radiality result} \label{sec_example}
In this final section we prove Theorems \ref{theo_example} and Lemma \ref{lemma_rad}. 
\begin{proof}[\textbf{Proof of Lemma \ref{lemma_rad}}]
Let  $f:B_1 \to \R$ be radial, $f \in H^1(B_1)$, and $g:[0,|B_1|] \to [0,+\infty)$   be   convex with $g(0)=0$ and $g>0$ in $(0,|B_1|]$.
Clearly 
\begin{multline}\label{ineq_rad}
\inf\left\{\frac{\int_{B_1} |\nabla w|^2 \, dx}{g(|\{w=f\}|)}:  w\in H^1_0(B_1),|\{w=f\}|\neq 0\right\}\\
\le \inf\left\{\frac{\int_{B_1} |\nabla w|^2 \, dx}{g(|\{w=f\}|)}:  w\in H^1_0(B_1),|\{w=f\}|\neq 0,w\text{ radial}\right\},
\end{multline}
thus we only need to prove the reverse inequality. 
Let $w\in H^1_0(B_1)$ with $g(|\{w=f\}|)>0$. The main idea is  to build a sequence of functions $w_k \in H^1_0(B_1)$ such 
that 
\begin{equation}
    \frac{1}{g(|\{w_k=f\}|)}\int_{B_1} |\nabla w_k|^2 \,dx \quad  \text{ is non-increasing},
\end{equation}
with $w_0:=w$ and $w_k$ is symmetrical with respect to $k$  hyperplanes.

To this end, for any $\nu \in \mb{S}^{d-1}$ let us consider the hyperplane $H_\nu:=\{x \in \R^d: \nu \cdot x=0\}$ and the half spaces 
$H^+_\nu:=\{x \in \R^d: \nu \cdot x>0\}$ and  $H^-_\nu:=\{x \in \R^d: \nu \cdot x<0\}$.
We define $w^+_\nu, w^-_\nu \in H^1_0(B_1)$ as 
\begin{align}
&w^+_\nu(x):=
\begin{cases}
w(x), &\text{ in } H^+_\nu \cap B_1,\\
w(x+2\mathop{\rm{dist}}(x,H_\nu) \nu ) &\text{ in } H^-_\nu \cap B_1,
\end{cases}\\
&w^-_\nu(x):=
\begin{cases}
w(x-2\mathop{\rm{dist}}(x,H_\nu) \nu ) &\text{ in } H^+_\nu \cap B_1,\\
w(x), &\text{ in } H^-_\nu \cap B_1,
\end{cases}
\end{align}
where $\mathop{\rm{dist}}(x,H_\nu)$ is the distance from the hyperplane $H_\nu$. 
Letting $\e:=|\{w=f\}|$, there exists $\sigma \in [0,1]$ such that 
\begin{equation}
|\{w=f\} \cap H_\nu^+|= \sigma \e \quad \text{ and  }  \quad |\{w=f\} \cap H_\nu^-|= (1-\sigma) \e, 
\end{equation}
since $f$ is radial.

If $\sigma=0$ then, since $g$ is convex and $g(0)=0$,
\begin{equation}
\frac{1}{g(2\e)}\int_{B_1} |\nabla w_{\nu}^-|^2 \,dx \le \frac{1}{2g(\e)}\int_{B_1} |\nabla w_{\nu}^-|^2 \,dx \le \frac{1}{g(\e)}\int_{B_1} |\nabla w|^2 \,dx.
\end{equation}
and similarly if $\sigma=1$
\begin{equation}
\frac{1}{g(2\e)}\int_{B_1} |\nabla w_{\nu}^+|^2 \,dx \le \frac{1}{2g(\e)}\int_{B_1} |\nabla w_{\nu}^+|^2 \,dx \le \frac{1}{g(\e)}\int_{B_1} |\nabla w|^2 \,dx.
\end{equation}
If $\sigma \in (0,1)$, we claim that either 
\begin{equation}
\frac{1}{g(2\e\sigma)}\int_{B_1} |\nabla w_{\nu}^{+}|^2 \,dx \le \frac{1}{g(\e)}\int_{B_1} |\nabla w|^2 \,dx \quad \text{ or } \quad \frac{1}{g((1-\sigma)\e)}\int_{B_1} |\nabla w_{\nu}^-|^2 \,dx \le \frac{1}{g(\e)}\int_{B_1} |\nabla w|^2 \,dx.
\end{equation}
Indeed, letting $a:=\int_{B_1} |\nabla w_{\nu}^{+}|^2 \,dx$ and $b:=\int_{B_1} |\nabla w_{\nu}^-|^2 \,dx$ the claim above reduces to check that
\begin{equation}
\frac{a}{g(2\e\sigma)} \le \frac{a+b}{2g(\e)} \quad \text{ or } \quad \frac{b}{g(\e(1-\sigma))} \le \frac{a+b}{2g(\e)},
\end{equation}
for any $a,b \in [0,+\infty)$, $\sigma \in (0,1)$ and $\e \in (0, |B_1|)$. Summing the two inequalities we obtain
\begin{equation}
2g(\e)\le g(2\e\sigma)+g(2\e(1-\sigma))
\end{equation}
which holds by convexity of $g$ since $g(\e)=g\left(\frac{2\e\sigma}{2}+\frac{2\e(1-\sigma)}{2}\right)$.

Hence, either $g(|\{w_\nu^+=f\}|)\neq 0$ or $g(|\{w_\nu^-=f\}|)\neq 0$ and 
\begin{equation}
\frac{\int_{B_1} |\nabla w_\nu^+|^2 \, dx}{g(|\{w^+_\nu=f\}|)} \le  \frac{\int_{B_1} |\nabla w|^2 \, dx}{g(|\{w_\nu=f\}|)} 
\quad \text{ or } \quad 
\frac{\int_{B_1} |\nabla w_\nu^-|^2 \, dx}{g(|\{w^-_\nu=f\}|)} \le  \frac{\int_{B_1} |\nabla w|^2 \, dx}{g(|\{w_\nu=f\}|)}.
\end{equation}
It follows that for any $w\in H^1_0(B_1)$ with $g(|\{w=f\}) \neq 0$ there exist a function $\psi \in H^1_0(B_1)$ with  $g(|\{\psi=f\} \neq 0$ even with respect to the hyperplane $H_\nu$ and with  
\begin{equation}
\frac{\int_{B_1} |\nabla \psi|^2 \, dx}{g(|\{\psi=f\}|)} \le  \frac{\int_{B_1} |\nabla w|^2 \, dx}{g(|\{w=f\}|)}.
\end{equation}
Let $\{v_i\}_{i \in \mathbb{N}\setminus \{0\}}\subset \mb{S}^{d-1}$ be a dense set in $\mb{S}^{d-1}$. By the arbitrariness of $\nu$ and iterating the same procedure a finite number of times, for  any $k \in \mathbb{N}$
\begin{multline}
\inf\left\{\frac{\int_{B_1} |\nabla w|^2 \, dx}{g(|\{w=f\}|)}: w\in H^1_0(B_1),  w\in H^1_0(B_1),|\{w=f\}|\neq 0\right\}\\
=\inf_{\e \in (0,|B_1|)}\inf\Bigg\{\frac{1}{g(\e)}\int_{B_1} |\nabla w|^2 \, dx: w\in H^1_0(B_1), |\{w=f\}|=\e,\\
w\text{ is even with respect to the hyperplanes } H_{\nu}, \text{ for } i=1, \dots, k\Bigg\}.
\end{multline}
Passing to the infimum with respect to $k$,
\begin{multline}
\inf\left\{\frac{\int_{B_1} |\nabla w|^2 \, dx}{g(|\{w=f\}|)}:  w\in H^1_0(B_1), |\{w=f\}|\neq 0\right\}\\
=\inf_{\e \in (0,|B_1|)}\inf_{k \in \mathbb{N}\setminus \{0\}}\inf\Bigg\{\frac{1}{g(\e)}\int_{B_1} |\nabla w|^2 \, dx: w\in H^1_0(B_1),  |\{w=f\}|=\e,\\
w\text{ is even with respect to the hyperplanes } H_{\nu}, \text{ for } i=1, \dots, k\Bigg\}.
\end{multline}
Let $\{w_k\}$ be a minimizing sequence of 
\begin{multline}
\inf_{k \in \mathbb{N}\setminus \{0\}}\inf\Bigg\{\frac{1}{g(\e)}\int_{B_1} |\nabla w|^2 \, dx:  |\{w=f\}|=\e,\\
w\text{ is even with respect to the hyperplanes } H_{\nu}, \text{ for } i=1, \dots, k\Bigg\}
\end{multline}
so that   $w_k \in H^1_0(B_1)$,  $g(|\{w_k=f\}|)=\e$,  and  $w_k$  is even with respect to the hyperplanes $H_{\nu}$, for $i=1, \dots, k$. Since,  $\{w_k\}$ is bounded in $ H^1_0(B_1)$, up to passing to a subsequence, there exists $w_\infty \in H^1_0(B_1)$ such that $w_k \rightharpoonup w_\infty$ weakly in $H^1_0(B_1)$ as $k \to \infty$ and $w_\infty$ is radial. Furthermore, since $g(|\{w_\infty=f\}|)\ge \liminf_{k} g(|\{w_k=f\}|)=\e$,
\begin{equation}
\frac{\int_{B_1} |\nabla w_\infty|^2 \, dx}{g(|\{w_\infty=f\}|)} \le  \liminf_{k \to \infty} \frac{\int_{B_1} |\nabla w_k|^2 \, dx}{g(|\{w_k=f\}|}.
\end{equation}
Hence, we conclude that 
\begin{multline}
\inf\left\{\frac{\int_{B_1} |\nabla w|^2 \, dx}{g(|\{w=f\}|)}:  w\in H^1_0(B_1),|\{w=f\}|\neq 0\right\}\\
=\inf_{\e \in (0,|B_1|)}\inf\Bigg\{\frac{1}{g(\e)}\int_{B_1} |\nabla w|^2 \, dx: w\in H^1_0(B_1), |\{w=f\}|=\e,w \text{ radial}\}\\
=\inf\Bigg\{\frac{\int_{B_1} |\nabla w|^2}{g(|\{w=f\}|)}:  w\in H^1_0(B_1),|\{w=f\}|\neq 0,w \text{ radial}\},
\end{multline}
thus we have proved \eqref{eq_rad}.
\end{proof}

We know turn to the proof of Theorem \ref{theo_example}.

\begin{proof}[\textbf{Proof of Theorem \ref{theo_example}}.]
Let  $d \ge3$.  For any $W \in H_0^1(B_1, \R^k)$, the pointwise inequality $ |\nabla W|^2 \ge|\nabla |W||^2$ implies that 
\begin{equation}
\Lambda^*(\mathop{\rm Id_d})\ge \inf_{\e\in(0, |B_1|)}\inf\left\{\frac{1}{\e}\int_{B_1} |\nabla w|^2 \, dx: w\in H^1_0(B_1), |\{w(x)=|x|\}|=\e\right\}.
\end{equation}
By Lemma \ref{lemma_rad} with $f(x)=|x|$ and $g(\e)=\e$, 
\begin{equation}
\Lambda^*(\mathop{\rm Id_d})\ge \inf_{\e\in(0, |B_1|)}\inf\left\{\frac{1}{\e}\int_{B_1} |\nabla w|^2 \, dx: w\in H^1_0(B_1), |\{w(x)=|x|\}|=\e, w\text{ radial}\right\}.
\end{equation}
For any fixed $\e \in (0, |B_1|)$,  letting $r_\e:=(\e/\omega_d)^{\frac{1}{d}}$, the only minimizer $w_\e$ of 
\begin{equation} 
\inf\left\{\frac{1}{\e}\int_{B_1} |\nabla w|^2 \, dx: w\in H^1_0(B_1), |\{w(x)=|x|\}|=\e, w\text{ radial}\right\}
\end{equation}
 is the function 
\begin{equation}
w_\e(x):= 
\begin{cases}
r_\e\frac{1-|x|^{2-d}}{1-r_\e^{2-d}},     & \text{ in  } B_1 \setminus B_{r_\e},\\
|x|,     & \text{ in  } B_{r_\e}.
\end{cases}
\end{equation}
Furthermore,
\begin{multline}
\int_{B_1} |\nabla w_\e|^2 \, dx= \e +   \frac{r_\e(2-d)^2}{(1-r_\e^{2-d})^2}d \omega_d \int_{B_1\setminus B_{r_\e}} |x|^{2-2d} \, dx\\
=\e +   \frac{r_\e(2-d)^2}{(1-r_\e^{2-d})^2}d \omega_d \int_{r_\e}^1 \rho^{1-d} \, d\rho = \omega_d r_\e^d\frac{1+d(d-2)}{1-r_\e^{d-2}},
\end{multline}
thus the function $h:(0,|B_1|) \to (0,+\infty)$ 
\begin{equation}
h(\e):=\frac{1}{\e}\int_{B_1} |\nabla w_\e|^2 \, dx=\frac{1+d(d-2)}{1-r_\e^{d-2}}
\end{equation}
is increasing. We conclude that 
\begin{equation}
\Lambda^*(\mathop{\rm Id_d}) \ge \inf_{\e \in (0,|B_1|)}h(\e)=\lim_{\e \to 0^+}h(\e) =1+d(d-2).
\end{equation}
\end{proof}

\section*{Acknowledgements}
\noindent
G. Siclari is partially supported by the 2026 INdAM--GNAMPA project 2026 ``Asymptotic analysis of variational problems''.
B. Velichkov was supported by the European Research Council's (ERC) project n.853404 ERC VaReg - \it Variational approach to the regularity of the free boundaries \rm, financed by the program Horizon 2020. 
B. Velichkov also acknowledges the MIUR Excellence Department Project awarded to the Department of Mathematics (CUP I57G22000700001) and also the support from the project MUR-PRIN “NO3” (n.2022R537CS). 

\bibliographystyle{acm}
\bibliography{references}	
\end{document}